\documentclass[11pt]{amsart}
\usepackage{graphicx}
\usepackage{amsmath,amsfonts,amsthm,amssymb}

\newtheorem{theorem}{Theorem}[section]
\newtheorem{lemma}{Lemma}[section]

\newtheorem{remark}{Remark}[section]
\newtheorem{definition}{Definition}[section]

\def\ds{\displaystyle}
\def\Adm{\mathbb{A}}
\def\bbR{\mathbb{R}}
\def\E{\mathbf{E}}
\def\H{\mathbf{H}}
\def\n{\mathbf{n}}
\def\u{\mathbf{u}}
\def\v{\mathbf{v}}
\def\x{\mathbf{x}}
\def\caC{\mathcal{C}}
\def\caD{\mathcal{D}}
\def\caG{\mathcal{G}}
\def\caK{\mathcal{K}}
\def\caT{\mathcal{T}}
\def\K{\text{K}}
\def\T{\text{T}}
\def\rot{{\nabla\times}}
\def\Einc{\mathbf{E}_\text{inc}}
\def\Hinc{\mathbf{H}_\text{inc}}
\def\Esc{\mathbf{E}_\text{sc}}
\def\Hsc{\mathbf{H}_\text{sc}}
\def\wt{\widetilde}
\newcommand\iref[1]{(\ref{#1})}
\DeclareMathOperator\Id{Id}

\begin{document}

\title[A simple preconditioned DDM for electromagnetic scattering problems]
{A simple preconditioned domain decomposition method for electromagnetic scattering problems}
\author{F.~Alouges}
\address[F.~Alouges]{CMAP, Ecole Polytechnique, Palaiseau, France}
\email[F.~Alouges]{alouges@cmapx.polytechnique.fr}
\author{J.~Bourguignon-Mirebeau}
\address[J.~Bourguignon-Mirebeau]{Laboratoire de Math\'ematiques, Universit\'e Paris-Sud XI, Orsay, France}
\email[J.~Bourguignon-Mirebeau]{jennifer.bourguignon@math.u-psud.fr}
\author{D.~P.~Levadoux}
\address[D.~P.~Levadoux]{DEMR-SFM, ONERA, Palaiseau, France}
\email[D.~P.~Levadoux]{david.levadoux@onera.fr}
\date{\today}
\maketitle

\begin{abstract}
We present a domain decomposition method (DDM)
devoted to the iterative solution of time-harmonic
electromagnetic scattering problems,
involving large and resonant cavities.
This DDM uses the electric field integral equation (EFIE)
for the solution of Maxwell problems
in both interior and exterior subdomains,
and we propose a simple preconditioner for the global method,
based on the single layer operator
restricted to the 
fictitious interface between the two subdomains. \\
\end{abstract}



\section{Introduction}

Solving scattering Maxwell problems
in harmonic regime can be achieved with various methods,
among which integral equations (which lead to the so-called boundary element methods) have proven their efficiency.
Their main advantage is that they allow to replace a problem posed on the whole space
by an equation posed on the surface of the scattering obstacle,
reducing a three-dimensional problem to a bi-dimensional one.
With the development of such methods,
several difficulties arose successively :
\begin{itemize}
\item These formulations lead classically to linear systems involving dense matrices
(in contrast with finite element methods, for instance). Several methods among 
which the most famous is probably the FMM (Fast Multipole Method) \cite{Rokhlin93}, \cite{SimonThesis} have been 
used to circumvent this difficulty.
\item There might exist irregular frequencies for which the problem is ill posed \cite{Nedelec01}.
This is typically the case for the so-called EFIE and MFIE formulations. Other types
of formulations (e.g. the CFIE) are instead well-posed at any frequency \cite{BurtonMiller71}.
\item The desire to deal with high frequency problems
imposes to use fine discretizations of the equations
and consequently to solve large linear systems.
This prevents the use of direct solvers, and one usually employs iterative methods.
On the one hand, this needs a fast matrix-vector multiplication (which is often realized  through the FMM),
while on the other hand iterative methods become sensible to the condition number of the system.
It has been shown that the underlying systems arising from integral equations are usually
badly conditioned and there is a need to develop preconditioning strategies
in order to accelerate the convergence of the iterative solver
\cite{ChristiansenNedelec02}, \cite{LevadouxM2AS07}, \cite{Steinbach98}, \cite{LeanTran97}. For instance,
the so-called GCSIE methodology has been developed which turns out to be particularly
efficient in the case where the object has no cavities and no singularities,
by building intrinsically well conditioned integral equations
\cite{Alouges07}, \cite{Alouges05}, \cite{Darbas2005}, \cite{LevadouxMillotPernet}, \cite{Pernet}.
\end{itemize}

Nevertheless, when facing realistic problems, one has to treat large objects with
complex geometries and new problems are encountered.
In this paper, we address the important issue of resonant cavities,
motivating the use of a domain decomposition method.
Indeed, this is a particularly crucial problem in stealth applications as one needs to take 
into account the existence of large and resonant cavities, such as air intakes, or cockpits  for aircrafts.
In classical numerical computations of radar cross sections, these cavities are usually closed
in order to avoid the poor convergence of the algorithms \cite{Alouges07} giving unrealistic results.

In this paper, we explore a new strategy to deal with this problem. Indeed, we intend to use
a domain decomposition method (DDM) in order to split the exterior domain into
two subdomains one of which being the cavity. The aim is to decouple the exterior
problem (without any cavity) from the problem with boundaries (the cavity itself). 
This introduces an artificial interface $\Sigma$ between these subdomains
and a new coupling problem posed on $\Sigma$ (\textsc{Fig.} \ref{picture1}).
For simplicity, we here use on each subdomain the EFIE to solve the corresponding
subproblems and to couple the solutions on $\Sigma$.
This naive DDM algorithm turns out to converge badly. In a latter part we propose a 
preconditioning technique to accelerate significantly the solution of the DDM.

Historically, the first domain decomposition methods
for Helmholtz or Maxwell problems
were applied using a finite element method (FEM) in the interior bounded subdomains
and a boundary element method (BEM) in the exterior unbounded domain.
For instance, Hiptmair considers FEM-BEM methods,
first applied to acoustic problems \cite{HiptmairHelmholtz} and then to electromagnetic problems \cite{HiptmairMaxwell2}.
For Helmholtz transmission problems, domain decomposition methods
have been used by Balin, Bendali and Collino \cite{BalinBendaliCollino}
to specifically treat the case of an electrically deep cavity,
and an integral preconditioner using the Calder{\'o}n formulas
has been developed by Antoine and Boubendir \cite{AntoineBoubendir2}.
For Maxwell transmission problems, Balin, Bendali and Millot \cite{BalinBendaliMillot} on the one hand,
Collino and Millot on the other hand \cite{CollinoMillot},  \cite{Collino00}
propose algebraic preconditioners which use overlapping or nonoverlapping domain decomposition techniques.
In iterative domain decomposition techniques, which are split into overlapping and nonoverlapping DDM,
the subdomains classically exchange Dirichlet or Neumann data.
A substantial improvement using absorbing boundary conditions is made by Despr\`es \cite{DespresWaves}, \cite{DespresThesis}.
The Schwarz method, originally used with Dirichlet or Neumann conditions for overlapping domains,
is then adapted by Gander, Halpern and Nataf \cite{Gander8},
to nonoverlapping subdomains with more general conditions, of Robin type.
The resulting algorithm converges with a high convergence rate for the wave equation in dimension 1.
Gander, Halpern and Magoul\`es \cite{Gander14} optimize the method by taking more general conditions, for the Helmholtz problem in dimension~2.
Eventually, Dolean, Gander and Gerardo-Giorda \cite{Gander17} adapt it to obtain
a Schwarz optimized method for the harmonic Maxwell problem in dimension~3.

\newpage

We present here a nonoverlapping domain decomposition method.
This DDM couples the subdomains through the help of an operator, instead of transmitting at each iteration
the appropriate conditions from one subdomain to another.
We use only integral equations to solve the boundary value problems in the subdomains.
In particular, the interior problem is treated with the help of an integral equation,
instead of a more classical finite element method.
We are not aware of the use of such techniques for solving Maxwell equations with the DDM in the context of integral equations  in the literature.

\begin{figure}[h!]
\begin{center}
\begin{minipage}{0.45\linewidth}
\includegraphics[scale=0.28]{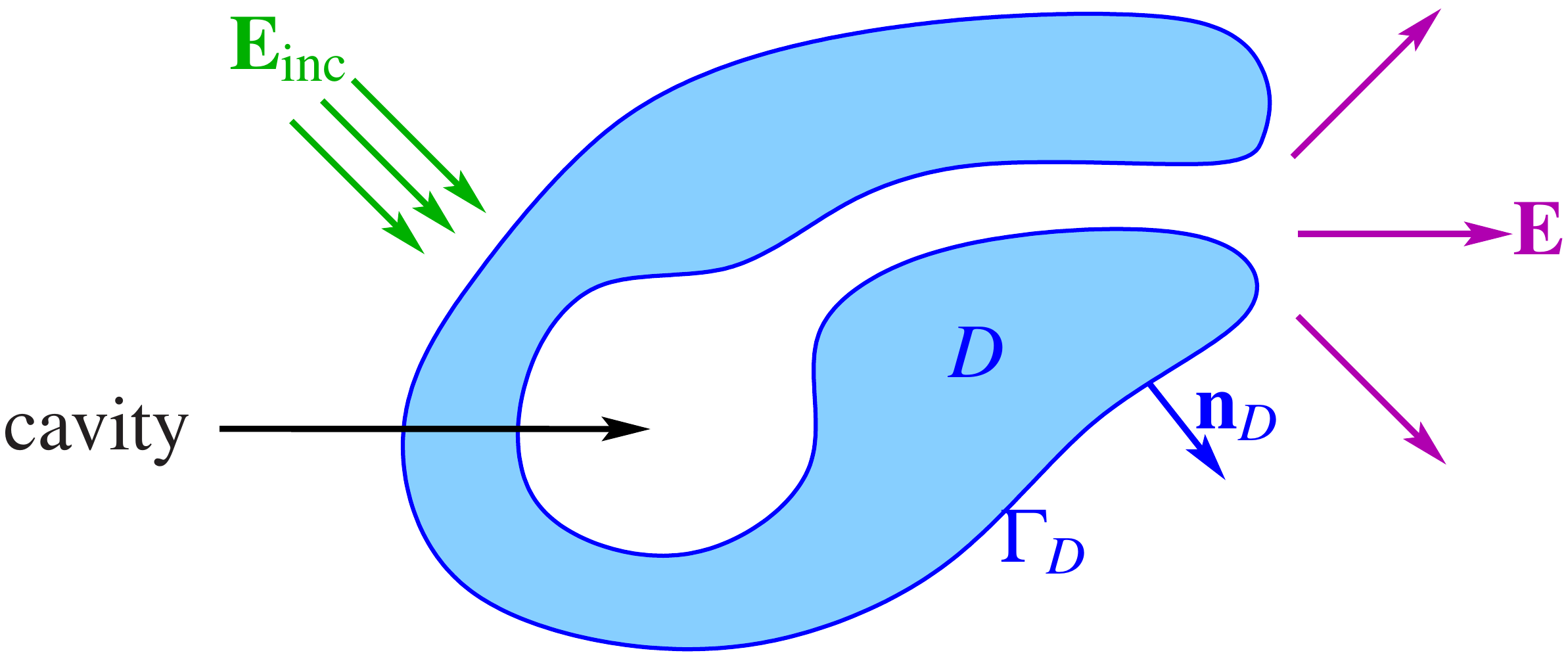}
\end{minipage}
\hfill
\begin{minipage}{0.45\linewidth}
\includegraphics[scale=0.28]{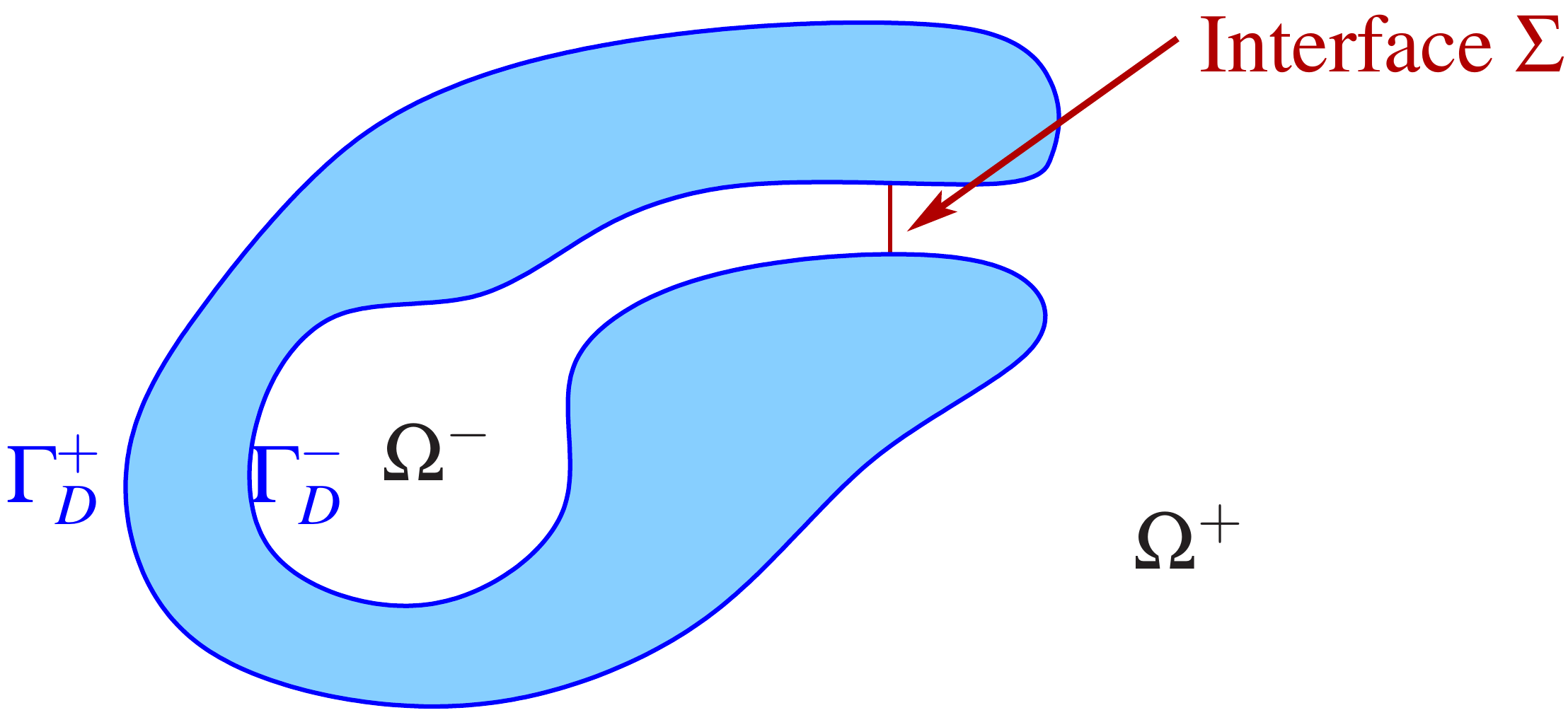}
\end{minipage}
\caption{
The scattering problem (left) and the decomposition of domain $\Omega$ (right).}
\label{picture1}
\end{center}
\end{figure}

The paper is organized as follows. The scattering problem is first described and in a second part, 
we present the domain decomposition method
and the condensed problem on the interface.
The third part gives a quick overview on classical integral equations methods,
and especially of the one we use here,
namely the EFIE (Electric Field Integral Equation).
The Dirichlet-to-Neumann map of the interface $\Sigma$ plays a very important role
that we describe carefully in the fourth part and this enables us to present a simple 
analytic preconditioner for the employed DDM, in the fifth part.
A validation of the method is presented using pseudo-differential calculus.
Eventually, the sixth part gives some numerical results. Substantial improvements are 
shown validating the approach.

\section{The boundary value problem : assumptions and notation}

We consider a compact set $D$
with a smooth boundary $\Gamma_D$.
We are particularly interested in the case where
the set $D$ contains a large cavity, as illustrated on \textsc{Fig.} \ref{picture1}.
We assume that the open exterior domain
$\Omega = \bbR^3 \setminus D$ is connected. \\

Our purpose is to solve the harmonic Maxwell problem
when $D$ stands for a scattering metallic object
\cite{ColtonKress83}, \cite{Nedelec01}.
Waves propagate with constant wave number $k$ in the exterior unbounded domain $\Omega$.
The electric field $\E$ is a vector-valued function which satisfies the harmonic Maxwell equation 
\begin{equation}
\label{Maxwell}
\rot \left( \rot \E \right) - k^2 \E = 0 \quad \text{ in } \Omega
\end{equation}
while the related magnetic field is given by 
\begin{equation}
\H = \frac{1}{ik} \rot \E.
\label{E->H}
\end{equation}
An electric field is said to be radiating if it satisfies the well-known
Sommerfeld radiation condition
\begin{equation}
\label{Sommerfeld}
\lim\limits_{\| \x \| \to \infty} \Big( \x \times \H + \| \x \| \E \Big) =  0.
\end{equation}

\newpage
Let $\E_{\text{inc}}$ be an incident electric field,
the electric field $\E$ scattered by the obstacle $D$
is the electric radiating field 
satisfying the boundary condition $\gamma_D \E = - \gamma_D \E_{\text{inc}}$ on $\Gamma_D$, 
where $\gamma_D = \n_D \times$ is the metallic trace on $\Gamma_D$, $\n_D$ being the 
unit normal outward to $D$.
In other words, the field $\E$ is solution of the following problem

\begin{equation}
\label{PECproblem}
\left\{ \begin{array}{rcll}
\ds \rot \left( \rot \E \right) - k^2 \E &=& 0 &\mbox{ in }\Omega,\\
\ds \gamma_D \E &=& - \gamma_D \E_{\text{inc}} &\mbox{ on }\Gamma_D,\\
\ds \lim\limits_{\| \x \| \to \infty} \Big( \x \times \left( \rot \E \right) + ik \| \x \| \E \Big) & = & 0, &
\end{array} \right.
\end{equation}
usually named as the perfect electric conductor (PEC) problem.

\section{Notation for the domain decomposition method}
\label{SectionDDM}

Domain decomposition methods rely on splitting the computational domain into several subdomains.
We present hereafter the application of the method for our case when $\Omega$ is decomposed into two subdomains
$\Omega^+$ and $\Omega^-$. Namely, we introduce an artificial boundary surface $\Sigma$, which splits $\Omega$
into an interior bounded domain $\Omega^-$ and an exterior unbounded domain $\Omega^+$
(\textsc{Fig.} \ref{picture1}).
We denote by $\Gamma_D^\pm = \Gamma_D \cap \partial \Omega^\pm$,
in such a way that the boundary of $\Omega^\pm$ is $\Gamma_D^\pm \cup \Sigma$.
We call $\mathbf{n}^\pm$ the inward unit normal to $\Omega^\pm$.
The notation $\sigma_0^\pm = \mathbf{n}^\pm \times$ and
$\sigma_1^\pm = \frac{1}{ik} \mathbf{n}^\pm \times \left( \nabla \times \right)$
stand for the classical electric and magnetic traces on $\Sigma$. \\

We introduce the short-cut field $\Esc$,
which is the radiating electric field defined on $\Omega^+$,
having a tangential trace on $\Gamma_D^+ \cup \Sigma$,
and such that $ \gamma_D \Esc=-\gamma_D \Einc $ on $\Gamma_D^+$
and $\sigma_0^+ \Esc = - \sigma_0^+ \Einc$ on $\Gamma$ (\textsc{Fig.} \ref{picture2}).
In other words, $\Esc$ is the field
scattered by the object when interface $\Sigma$ becomes metallic. \\

\begin{figure}[h!] \centering
\includegraphics[scale=0.28]{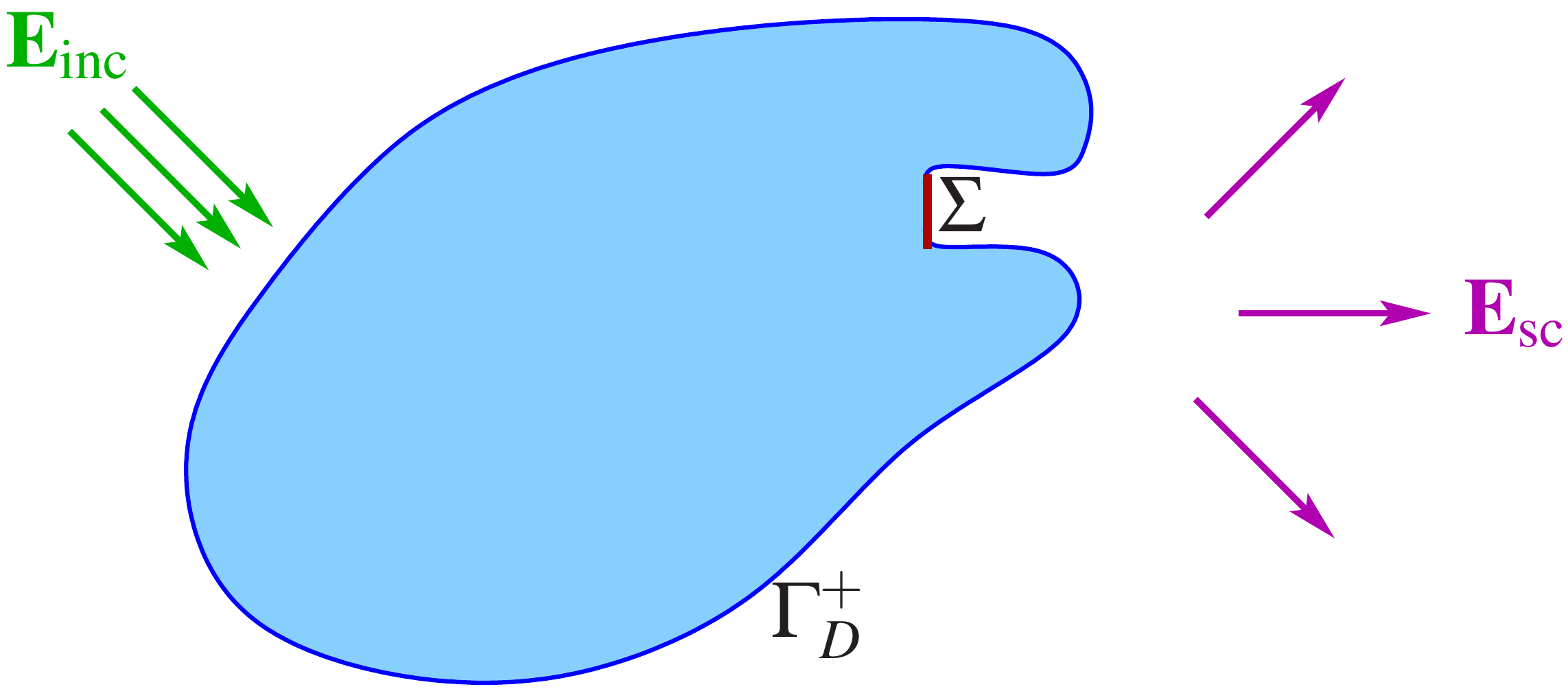}
\caption{
Short-cut field $\E_\text{sc}$. }
\label{picture2}
\end{figure}

We denote by $\Einc^\pm$ the restriction of the incident field $\Einc$ to the domain $\Omega^\pm$.
We look for the scattered field solution of the PEC problem \iref{PECproblem}
under the form $\E^- - \Einc^-$ inside $\Omega^-$
and $\E^+ + \Esc$ inside $\Omega^+$,
where $\E^+$ and $\E^-$ respectively belong to spaces of \textit{admissible} waves $W^+$ and $W^-$. \\

More precisely, the space $W^-$ is the set of all electric fields $\E^-$ 
which are defined on $\Omega^-$,
have a tangential trace on $\Gamma_D^- \cup \Sigma$,
and satisfy $\gamma_D \E^- = 0$ on $\Gamma_D^-$.
Correspondingly, $W^+$ is the space of all \emph{radiating} electric fields $\E^+$
which are defined on $\Omega^+$,
have a tangential trace on $\Gamma_D^+ \cup \Sigma$,
and satisfy $\gamma_D \E^+ = 0$ on $\Gamma_D^+$.
Since the subdomain $\Omega^-$ is bounded,
the radiation condition is not required for the fields in $W^-$. \\

The total electric fields therefore have the expression
\begin{equation*}
\E^\text{tot} = \left\{
\begin{array}{ll}
\E^- & \text{ in } \Omega^-,\\
\E^+ + \Einc^+ + \Esc & \text{ in } \Omega^+,
\end{array}
\right.
\end{equation*}
whereas the total magnetic fields (computed from the electric fields with \iref{E->H}) have the expression
\begin{equation*}
\H^\text{tot} = \left\{
\begin{array}{ll}
\H^- & \text{ in } \Omega^-,\\
\H^+ + \Hinc^+ + \Hsc & \text{ in } \Omega^+.
\end{array}
\right.
\end{equation*}

\begin{figure}[h!] \centering
\includegraphics[scale=0.28]{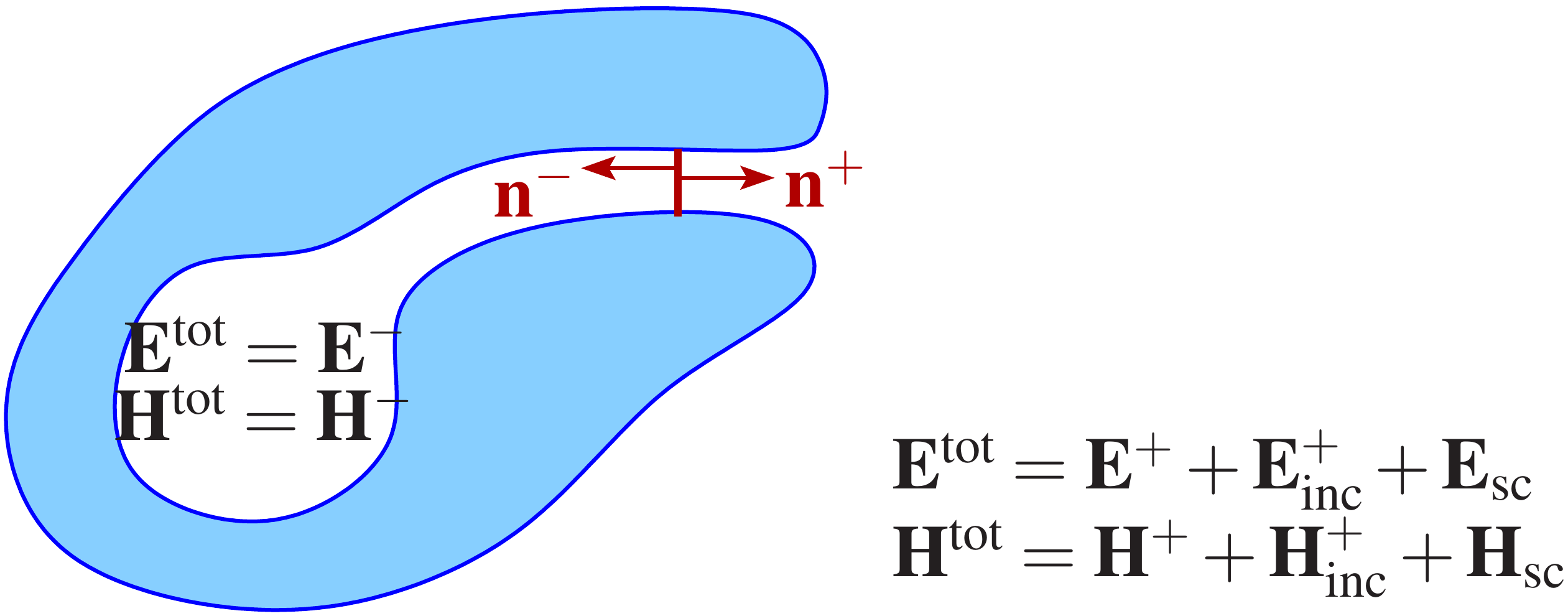}
\caption{
Total fields $\E^\text{tot}$ and $\H^\text{tot}$. }
\label{picture3}
\end{figure}

Since $\Sigma$ is an artificial boundary, the total fields $\E^\text{tot}$ and $\H^\text{tot}$
are continuous across $\Sigma$, and the problem~\iref{PECproblem} becomes the transmission problem
\begin{equation}
\label{Transmission1}
\text{Find } \left( \E^+, \E^- \right) \in (W^+,W^-), \quad
\begin{cases}
\n^+ \times \E^- = \n^+ \times \Big( \E^+ + \Einc^+ + \Esc \Big), \\
\n^+ \times \H^- = \n^+ \times \Big( \H^+ + \Hinc^+ + \Hsc \Big),
\end{cases}
\text{ on } \Sigma.
\end{equation}

Notice that by construction the short-cut field $\Esc$ verifies 
$\n^+ \times \Einc^+ + \n^+ \times \Esc = 0$ on $\Sigma$, and thus,
defining the right hand side current
\begin{equation*}
\u_\text{rhs} = \n^+ \times \Hinc^+ + \n^+ \times \Hsc \quad \text{ on } \Sigma,
\end{equation*}
equation \iref{Transmission1} rewrites as
\begin{equation}
\label{Transmission2}
\begin{cases}
\E_\text{tan}^- = \E_\text{tan}^+, \\
- \n^- \times \H^- = \n^+ \times \H^+ + \u_\text{rhs},
\end{cases}
\text{ on } \Sigma.
\end{equation}

The preceding system, which expresses the DDM,
will be solved using integral equations inside each subdomain.
We recall these integral equations methods in the next section.

\section{Integral equations}

Integral equations methods are commonly used to solve
electromagnetic scattering problems.
We hereafter give a short overview
of the construction of these methods.
We first recall the definitions of the single and double layer potentials,
as well as the fundamental Stratton-Chu formula,
before describing the principle of those integral equations. \\

Let $D_0$ be a compact and connected subset of $\bbR^3$
with a smooth boundary $\Gamma_0$,
defining two open and connected domains,
the interior bounded domain $\Omega_0^-$
and the exterior unbounded domain $\Omega_0^+$.
We denote by $\n$ the unit outward normal to $\Gamma_0$
and by $\gamma_T$ the tangential trace on $\Gamma_0$
from domain $\Omega_0^+$.
The PEC problem can be formulated as follows:
find the electric radiating field $\E$
defined on $\Omega_0^+$
and satisfying the boundary condition $\gamma \E = \u_0$,
where $\gamma = \n \times \gamma_T$ is a trace on $\Gamma_0$,
and $\u_0 = - \n \times \Einc$ is a given current depending on an incident field. \\

The classical vector potential $\caG$
maps a tangential vector-field $\u \in \caD_T'(\Gamma_0)$
to the vector-field defined on $\Omega_0^+$ and $\Omega_0^-$ by
\begin{equation}
\label{GreenPotential1}
\caG \u (x) = - \ \frac{1}{4 \pi} \int_{\Gamma_0}
\frac{e^{ik \| x-y \|}}{\| x-y \|} \u (y) \ dy,
\end{equation}
where $\|.\|$ denotes the euclidean norm on $\bbR^3$.
Then we define the single layer potential $\caT$ and the double layer potential $\caK$ by
\begin{equation}
\label{TKpotentials1}
\caT = \frac{1}{ik} \rot \left( \rot \caG \right) \quad \text{ and } \quad \caK = \rot \caG.
\end{equation}
The electromagnetic potentials satisfy the following important property:
given a current $\u$ on $\Gamma_0$,
the fields $\caT \u$ and $\caK \u$
are automatically solutions of Maxwell equation \iref{Maxwell}
and the radiation condition \iref{Sommerfeld}, \cite{Nedelec01}.
The boundary operators $\T$ and $\K$ are obtained from the electromagnetic potentials
and are defined by
\begin{equation}
\label{TracesTK}
\n \times \T = \n \times \gamma_T (\caT)
\quad \text{ and } \quad
\n \times \K = \n \times \gamma_T (\caK) + \Id / 2.
\end{equation}
It turns out that $\T$ and $\K$ are pseudo-differential operators
respectively of order $+1$ and $-1$ \cite{chazarain}, \cite{ColtonKress83}, \cite{Nedelec01}. \\

The Stratton-Chu formulas \cite{ColtonKress83}, \cite{Nedelec01} use the single and double layer potentials
to express an electric radiating field $\E$ and the related magnetic field $\H$
in terms of their boundary traces.
\begin{equation}
\label{RepresentationEH}
\E= \caT (\n \times \H) - \caK (\n \times \E)
\quad \text{ and } \quad \H= - \caK (\n \times \H) - \caT (\n \times \E).
\end{equation}
These formulas, also known as representation theorem,
are the foundation of integral equations, as we shall see now. \\

The incident field does not satisfy the radiation condition
and therefore the representation theorem \iref{RepresentationEH}
does not apply to $\Einc$ and $\Hinc$.
Instead, because $\Einc$ and $\Hinc$ are continuous on the whole space $\bbR^3$,
their traces have no jump across $\Gamma_0$ and one can show that
\begin{equation}
\label{RepresentationEinc}
0 = \caT (\n \times \Hinc) - \caK (\n \times \Einc)
\quad \text{and} \quad  0 = - \caK (\n \times \Hinc) - \caT (\n \times \Einc).
\end{equation}
Summing up \iref{RepresentationEH} and \iref{RepresentationEinc},
combined with the PEC boundary condition,
we obtain the EFIE and MFIE equations,
\begin{equation}
\label{EFIEandMFIE}
\text{EFIE : } \T \u = - \Einc^\text{tan},
\quad \text{MFIE : } \left( \n \times \K + \frac{1}{2} \Id \right) \u = \n \times \Hinc,
\end{equation}
where $\Einc^\text{tan}$ is the tangential component of $\Einc$
and the unknown $\u$ is equal to $\n \times \gamma_T (\H + \Hinc)$. \\

Unfortunately, the EFIE and the MFIE are well-known to be ill posed at resonant frequencies~\cite{Nedelec01}.
Their linear combination, weighted by an arbitrary parameter $\alpha \in ]0,1[$,
yields the CFIE, which instead is well posed at any frequency \cite{Mitzner68}, \cite{BurtonMiller71},
\begin{equation}
\label{CFIE}
\text{CFIE : } (1-\alpha) \T \u + \alpha \left( \n \times \K + \frac{1}{2} \Id \right) \u
= - (1-\alpha) \Einc^\text{tan} + \alpha \n \times \Hinc.
\end{equation}
In what follows and for the sake of simplicity, since we are mainly interested in the 
interface problem on $\Sigma$, we only concentrate on the EFIE for solving 
the electromagnetic problems inside the subdomains.

\section{Admittance operators and the DDM}
\label{SectionAdmittance}

The structure of the problem \iref{Transmission2}
naturally leads to introduce the so-called admittance operators\footnote{Such operators are
also classically called Dirichlet-to-Neumann or Steklov-Poincar\'e operators.} on $\Sigma$
\begin{equation*}
\Adm_\Sigma^\pm: \E_\text{tan}^\pm \mapsto \n^\pm \times \H^\pm,
\end{equation*}
where $\H^\pm=\frac{1}{ik}\rot \E^\pm$, and $\E^\pm\in W^\pm$ solves $\gamma_T \E^\pm=\E_\text{tan}^\pm$.
Notice that although the input and output data of $\Adm_\Sigma^\pm$ are defined only on $\Sigma$, the admittance operators 
are highly non-local and depend on the whole geometry of the domains $\Omega^\pm$. \\

We remark that 
\begin{equation*}
\Adm_\Sigma^\pm = R^\pm \Adm^\pm P^\pm,
\end{equation*}
where $P^\pm: \caD'(\Sigma) \to \caD'(\Gamma_D^\pm \cup \Sigma)$ extends by 0, on $\partial \Omega^\pm$,
data defined on $\Sigma$, while conversely, $R^\pm: \caD'(\Gamma_D^\pm \cup \Sigma) \to \caD'(\Sigma)$
restricts to $\Sigma$ data defined on $\partial \Omega^\pm$.
Here $\Adm^\pm$ are the admittance operators of $\partial \Omega^\pm$ which map currents $\E_\text{tan}^\pm$
defined on the whole {\em closed} boundaries $\partial \Omega^\pm$ to their magnetic traces $\n^\pm \times \H^\pm$.\\

The system \iref{Transmission2} can then be expressed
in terms of the admittance operators  $\Adm_\Sigma^\pm$
\begin{equation}
\label{Transmission3}
\begin{cases}
\E_\text{tan}^- = \E_\text{tan}^+, \\
- \Adm_\Sigma^- \E_\text{tan}^- = \Adm_\Sigma^+ \E_\text{tan}^+ + \u_\text{rhs},
\end{cases}
\end{equation}
which eventually reduces to
\begin{equation}
\label{DDM}
(\Adm_\Sigma^+ + \Adm_\Sigma^-) \E_\text{tan} = - \u_\text{rhs}, \\
\end{equation}
with $\E_\text{tan} = \E_\text{tan}^+ = \E_\text{tan}^- $. \\

Equation \iref{DDM} is at the heart of our domain decomposition method.
We explain below how the admittance operators $\Adm_\Sigma^+$ and $\Adm_\Sigma^-$ are numerically computed,
while Section \ref{SectionPreconditioner}
is devoted to the preconditioning of the subsequent linear system. \\

The admittance operators $\Adm_\Sigma^\pm$ can be naturally obtained by solving an integral equation,
which involves four electromagnetic potentials described below.
The main difference between our particular case
and the classical theory is that the domains $\Omega^+$ and $\Omega^-$
are \emph{not complementary one to another}.
The boundaries of these domains are therefore distinct,
although they share the same interface $\Sigma$.
Consequently, the convolution operators with the Green kernel
related to the exterior and the interior electromagnetic potentials,
that we next introduce,
are not defined on the same surfaces. \\

Similarly to the potentiel $\caG$ defined by \iref{GreenPotential1},
we define the vector potentials $\caG^\pm$
which map tangential vector-fields $\u^\pm \in \caD_T' \left( \Gamma_D^\pm \cup \Sigma \right)$
to the vector-fields defined on $\Omega^\pm$ by
\begin{equation*}
\caG^\pm \u^\pm (x) = - \ \frac{1}{4 \pi} \int_{\Gamma_D^\pm \cup \Sigma}
\frac{e^{ik \| x-y \|}}{\| x-y \|} \u^\pm (y) \ dy.
\end{equation*}
As before, the potentials $\caG^\pm$ are used to define
the single layer potentials $\caT^\pm$ and the double layer potentials $\caK^\pm$ as follows,
\begin{equation*}
\caT^\pm = \frac{1}{ik} \rot \left( \rot \caG^\pm \right) \quad \text{ and } \quad \caK^\pm = \rot \caG^\pm,
\end{equation*}
while
\begin{equation*}
\n^\pm \times \T^\pm = \n^\pm \times \gamma_T^\pm (\caT^\pm)
\quad \text{ and } \quad
\n^\pm \times \K^\pm = \n^\pm \times \gamma_T^\pm (\caK^\pm) + \Id / 2,
\end{equation*}
where $\gamma_T^\pm$ stand for the tangential traces on $\partial \Omega^+$ and $\partial \Omega^-$.
As previously, as long as the boundaries $\Gamma_D^\pm \cup \Sigma$ are smooth,
pseudo-differential operators $\T^\pm$ and $\K^\pm$ are of order $+1$ and $-1$ respectively. \\

Given an electric field $\E^\pm \in W^\pm$ and its magnetic counterpart $\H^\pm$,
we define the following electromagnetic traces on the boundary $\Gamma_D^\pm \cup \Sigma$,
\begin{equation}
\label{EMtraces}
\sigma_0^\pm \E^\pm = \n^\pm \times \E^\pm \quad \text{ and } \quad
\sigma_1^\pm \E^\pm = \n^\pm \times \H^\pm \quad \text{ on } \partial \Omega^\pm = \Gamma_D^\pm \cup \Sigma.
\end{equation}

Using the representation theorem \iref{RepresentationEH}
in the domain $\Omega^\pm$,
we obtain the expression of any electric field $\E^\pm$ in $W^\pm$
in terms of its boundary traces on $\Gamma_D^\pm \cup \Sigma$.
\begin{equation}
\label{Stratton-Chu2}
\forall \E^\pm \in W^\pm, \quad
\E^\pm = \caT^\pm (\sigma_1^\pm \E^\pm) - \caK^\pm (\sigma_0^\pm \E^\pm).
\end{equation}

Now, for a current $\u_0^\pm \in \caD'_T(\Sigma)$ defined on the fictitious interface $\Sigma$,
we have $\Adm_\Sigma^\pm \u_0^\pm = \sigma_1^\pm (\E^\pm)$
where $\E^\pm \in W^\pm$ is such that  $\E_\text{tan}^\pm = P^\pm \u_0^\pm$.
Several integral formulations can be used to compute effectively $\Adm_\Sigma^\pm \u_0^\pm$. As 
an example, we shall see hereafter that $\Adm_\Sigma^\pm \u_0^\pm=R^\pm \u^\pm$ where
$\T^\pm (\u^\pm) = \left( \frac{1}{2} \Id + \K^\pm \n^\pm \times \right) ( P^\pm \u_0^\pm)$. \\

Indeed, restricting ourselves to the exterior subdomain,
our first goal is to find $\E^+ \in W^+$ such that $\E_\text{tan}^+ = P^+ \u_0^+$.
Applying the trace $\sigma_0^+$ to the Stratton-Chu formula \iref{Stratton-Chu2}
leads to
\begin{equation*}
(\n^+ \times \T^+) (\sigma_1^+ \E^+) = \left( \frac{1}{2} \Id + \n^+ \times \K^+ \right) (\sigma_0^+ \E^+)
\quad \text{ on } \partial \Omega^+.
\end{equation*}
Taking the cross product of the previous equation with $- \n^+$ yields
\begin{equation*}
\T^+ (\n^+ \times \H^+) = \left( \frac{1}{2} \Id + \K^+ \n^+ \times \right) (\E_\text{tan}^+)
\quad \text{ on } \partial \Omega^+.
\end{equation*}
This problem has the form of
an electric field integral equation (EFIE):
\begin{equation}
\label{EFIE}
\text{Find the current } \u^+, \text{ such that } \quad
\T^+ (\u^+) = \left( \frac{1}{2} \Id + \K^+ \n^+ \times \right) ( P^+ \u_0^+)
\quad \text{ on } \partial \Omega^+.
\end{equation}
The restriction to $\Sigma$ of the solution $\u^+$ of \iref{EFIE} is the trace $\n^+ \times \H^+$ we look for,
and we therefore have
\begin{equation*}
\Adm_\Sigma^+ \u_0^+ = R^+ \u^+,
\end{equation*}
as claimed. \\

The cavity ($\Omega^-$) is treated similarly, with the restriction that $\Adm_\Sigma^-$
is well defined. This is the case when $k^2$ is not an eigenvalue for the interior Maxwell problem.

\begin{remark}
Since this method is based on an EFIE formulation, it applies to the computation of the admittance $\Adm_\Sigma^\pm$ in both subdomains,
but with the following \emph{caveat}:
metallic problems having irregular frequencies (\cite{ColtonKress83}, \cite{Nedelec01}),
the EFIE is ill-posed at frequencies close to these resonances.
Despite this drawback, the EFIE is still widely used for it is one of the methods which give the most accurate results.
\end{remark}

\section{Preconditioning the DDM}
\label{SectionPreconditioner}

As we shall see in Section \ref{Results}, the equation \iref{DDM} unfortunately leads after discretization to an ill-conditioned linear system.
We therefore propose a simple preconditioner 
in order to obtain a tractable DDM which improves the convergence rate. \\

The purpose of this section
is to introduce a theoretical framework
which suggests that the preconditioned equation is well posed.
Unfortunately, these theoretical results only apply so far in an ideal setting
which is not satisfied in practical situations.
They should therefore be regarded as
a heuristic and hopefully as the foundation of future more general results.

\subsection{A preliminary lemma}

Let $D_0$ be a compact subset of $\bbR^3$
with a smooth boundary $\Gamma_0$,
defining two open domains :
the interior domain $\Omega_0^-$ (the interior of $D_0$)
and the exterior domain $\Omega_0^+ = \bbR^3 \setminus D_0$.
We define as before the admittance operators
related to the boundary $\Gamma_0$,
$\Adm_0^-$ for the interior domain
and $\Adm_0^+$ for the exterior domain.
Also, the single layer potential $\caT$
and its tangential trace $\emph{T}$ are
respectively given by  \iref{TKpotentials1} and \iref{TracesTK}
for the boundary $\Gamma_0$. \\

The idea for preconditioning our method is based on the
following lemma.

\begin{lemma}
\label{LemExact}
If $k^2$ is not an eigenvalue for the interior Maxwell problem, then
$\Adm_0^-$ is well defined and we have
\begin{eqnarray}
\label{RightPrec1}
(\Adm_0^+ + \Adm_0^-) \emph{T} & = & \Id , \\
\label{LeftPrec1}
\quad \emph{T} (\Adm_0^+ + \Adm_0^-) & = & \Id.
\end{eqnarray}
\end{lemma}

\begin{proof}
Let $\u \in \caD_T'(\Gamma_0)$ be a current on $\Gamma_0$.
We define $\E = \caT \u$ on $\Omega^+$ and $\Omega^-$.
Then by continuity of the potential $\caT$ across $\Gamma_0$,
\begin{equation*}
\E_\text{tan}^+ = \E_\text{tan}^- = \T \u.
\end{equation*}
Thus
\begin{eqnarray*}
\n^+ \times \H^+ +\n^- \times \H^-
& = & \Adm_0^+ \E_\text{tan}^+ + \Adm_0^- \E_\text{tan}^- = (\Adm_0^+ + \Adm_0^-) (\T \u) \\
& = & (\sigma_1^+ \caT - \sigma_1^- \caT) \u = \u,
\end{eqnarray*}
since the Neumann gap of the single layer potential is the identity \cite{Nedelec01}.
We recall that $\sigma_1^\pm$ is the electromagnetic trace defined in \iref{EMtraces}.
We have proven \iref{RightPrec1}. \\

For \iref{LeftPrec1}, since $k^2$ is not an eigenvalue for the interior Maxwell problem,
the operator $\T$ is bijective
and there exists $\v \in \caD'_T(\Gamma_0)$ such that $\u = \T \v$.
Defining $\E = \caT \v$ on $\Omega^+$ and $\Omega^-$ leads to
\begin{equation*}
\E_\text{tan}^+ = \E_\text{tan}^- = \T \v = \u.
\end{equation*}
Consequently,
\begin{equation*}
\T (\Adm_0^+ + \Adm_0^-) \u
= \T (\sigma_1^+ - \sigma_1^-) \caT \v
= \T \v = \u,
\end{equation*}
since the Neumann gap of the single layer potential is the identity.
\end{proof}

This lemma suggests to precondition the equation \iref{DDM}
by the operator $\T_\Sigma$ defined by the operator
$\T$ restricted to the interface $\Sigma$.

\begin{definition}[Operator $\T_\Sigma$] \;\ \\
We denote by $F$ the explicit kernel of the single layer potential $\caT$,
which can be computed from \iref{GreenPotential1} and \iref{TKpotentials1}.
The operator $\textnormal{T}_\Sigma: \caD'_T(\Sigma) \to \caD'_T(\Sigma)$ is defined as the convolution operator,
restricted to $\Sigma$
\begin{equation*}
\textnormal{T}_\Sigma \u(x) = \int_{\Sigma} F(x-y) \u(y) \ dy.
\end{equation*}
Notice that $\textnormal{T}_\Sigma$ does not depend on $\Gamma_D^\pm$.
\end{definition}

\subsection{A preconditioner for the DDM}

We aim at proving that the operators
$\T_\Sigma (\Adm_\Sigma^+ + \Adm_\Sigma^-)$
and $(\Adm_\Sigma^+ + \Adm_\Sigma^-) \T_\Sigma$
are compact perturbations of the identity,
using arguments of pseudo-differential theory.
This unfortunately restricts our results to smooth boundaries $\partial \Omega^\pm$
which in turn implies that $\partial \Omega$ is not smooth in general (see \textsc{Fig.}~\ref{picture1}). \\

Therefore, we consider a simplified setting in which we assume 
that the boundary $\Gamma_D$ is not smooth but such that both boundaries
$\Gamma_D^+ \cup \Sigma$ and $\Gamma_D^- \cup \Sigma$ are of $\caC^\infty$ regularity.
For instance, in dimension 2, this implies the existence of two cusps (\textsc{Fig.}~\ref{picture4}, on the left ). \\

\begin{figure}[h!]
\begin{center}
\begin{minipage}{0.45\linewidth}
\includegraphics[scale=0.35]{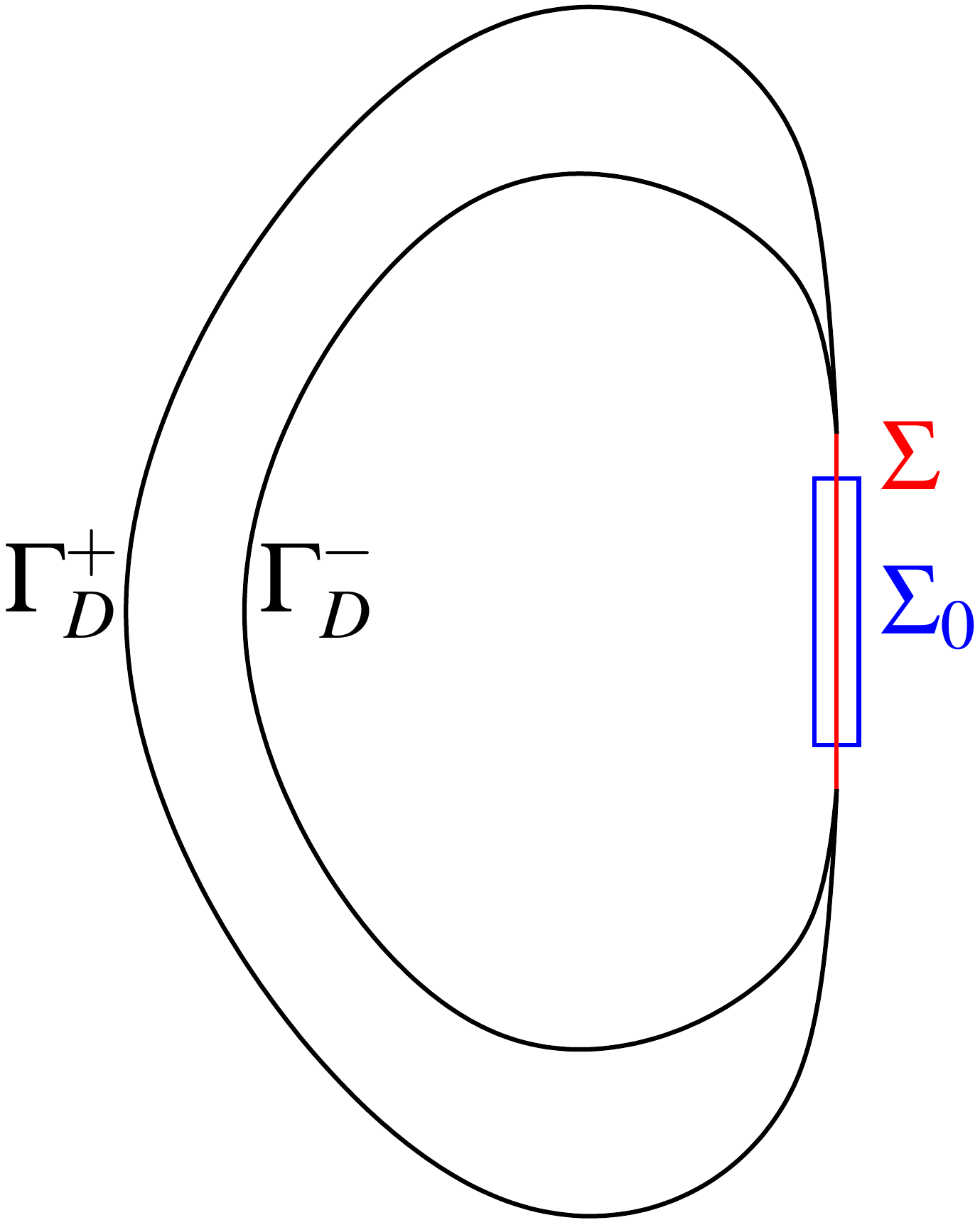}
\end{minipage}
\hfill
\begin{minipage}{0.45\linewidth}
\includegraphics[scale=0.35]{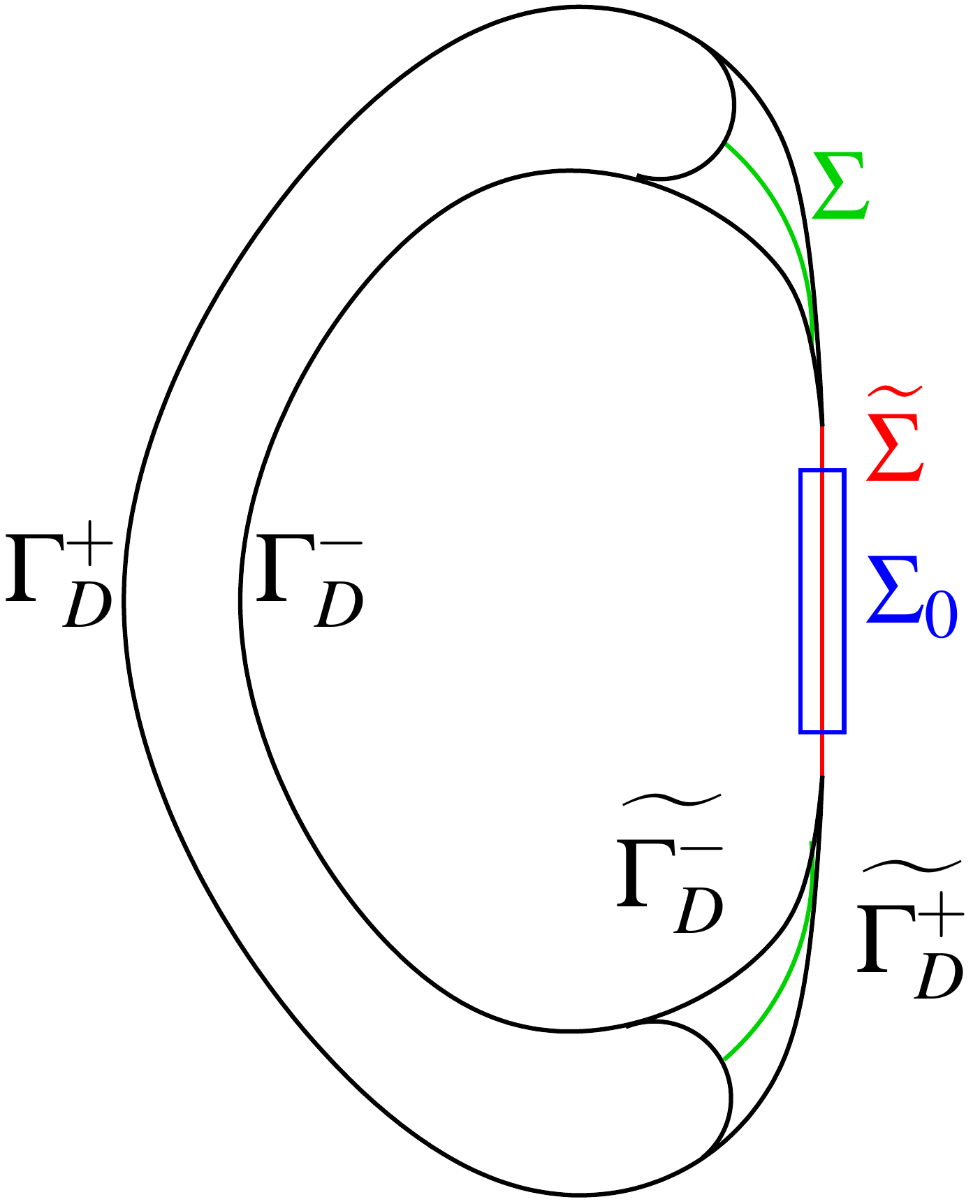}
\end{minipage}
\caption{
The setting for the theorem (left) and the real case (right). }
\label{picture4}
\end{center}
\end{figure}

Let $\Sigma_0$ be a compact subset of $\Sigma$,
and let $\chi_0: \Sigma \to [0,1]$ be
a $\caC^\infty$ cut-off function,
supported in the interior of $\Sigma$,
and such that  $\chi_0 = 1$ on $\Sigma_0$.
We denote by $\T_\Sigma = R^\pm \T^\pm P^\pm$
and by $\wt{\T_\Sigma}$ the operator
\begin{equation}
\label{Ttilde}
\wt{\T_\Sigma} = \chi_0 \T_\Sigma.
\end{equation}

We prove that $\wt{\T_\Sigma}$ is a good \emph{left} preconditioner
when applied to functions supported on $\Sigma_0$.

\newpage

\begin{theorem}[A preconditioner for the DDM] \;\ \\
\label{Theorem}
Let $\H_{T,0}^s (\Sigma) = \{ \u \in \H_T^s(\Sigma), \text{ such that } \u = 0 \text{ on } \Sigma \setminus \Sigma_0 \}$.
For all $\u \in \H_{T,0}^s(\Sigma)$, we have
\begin{equation}
\label{LeftPrec2}
\wt{\emph{T}_\Sigma} (\Adm_\Sigma^+ + \Adm_\Sigma^-) \u = \u + \v,
\end{equation}
where $\v \in \H_T^{s+1}(\Sigma)$.
More precisely, $\wt{\emph{T}_\Sigma} (\Adm_\Sigma^+ + \Adm_\Sigma^-)$
is a compact perturbation of the identity in $\H_{T,0}^s(\Sigma)$.
\end{theorem}

\begin{proof}
The representation theorem \iref{Stratton-Chu2} applied to the smooth boundary
$\Gamma_D^+ \cup \Sigma$  yields
\begin{equation*}
(\n^+ \times \T^+) (-\Adm^+ \n^+ \times) - \left( \n^+ \times \K^+ - \frac{1}{2} \Id \right) = \Id.
\end{equation*}
Therefore, 
\begin{equation}
\label{TA}
\T^+ \Adm^+ = \frac{1}{2} \Id + \K^+ \n^+ \times.
\end{equation}
Since the boundary $\Gamma_D^+ \cup \Sigma$ is smooth,
the operator $\K^+ \n^+ \times$ is of order $-1$.
Our goal is to extend this result to the operator $\wt{\T_\Sigma} \Adm_\Sigma^+$.
We have
\begin{eqnarray*}
\wt{\T_\Sigma} \Adm_\Sigma^+ & = & \left( \chi_0 R^+ \T^+ P^+ \right) \left( R^+ \Adm^+ P^+ \right) \\
& = & \chi_0 R^+ (\T^+ \Adm^+) P^+
+ \chi_0 R^+ \T^+ \left(P^+ R^+ \Adm^+ - \Adm^+ \right) P^+.
\end{eqnarray*}

It is clear to see that for $\u \in \H_{T,0}^s(\Sigma)$,
$\v := \left(P^+ R^+ \Adm^+ - \Adm^+ \right) P^+ \u = (P^+ R^+ - \Id) \Adm^+ P^+ \u$ vanishes on $\Sigma$.
Since $\T^+$ is a convolution operator with a kernel $F(x,y)$
which is $\caC^\infty$ for $x \neq y$,
we have $\chi_0 R^+ \T^+ \v \in \H_T^\infty (\Sigma)$.
This shows that the operator
\begin{equation*}
D^+ = \chi_0 R^+ \T^+ \left(P^+ R^+ \Adm^+ - \Adm^+ \right) P^+
\end{equation*}
is of order $- \infty$.
Note that we have used the fact  that the support of $\chi_0$
is included in the interior of $\Sigma$. \\

On the other hand, \iref{TA} leads to
\begin{equation*}
\chi_0 R^+ (\T^+ \Adm^+) P^+ = \chi_0 R^+ \left( \frac{1}{2} \Id + \K^+ \n^+ \times \right) P^+
= \chi_0 \frac{1}{2} \Id + \chi_0 R^+ (\text{K}^+ \n^+ \times) P^+.
\end{equation*}
Notice that $\ds \chi_0 \frac{1}{2} \Id = \frac{1}{2} \Id$ in $\H_{T,0}^s(\Sigma)$
and that $\chi_0 R^+ (\text{K}^+ \n^+ \times) P^+$ is a pseudo-differential operator of order $-1$ in $\H_{T,0}^s(\Sigma)$,
and is therefore compact.
Remark that we have used the fact that $P^+ (\H_{T,0}^s(\Sigma)) \subset \H_T^s (\partial \Omega^+)$. \\

Having the same results for the interior case, we obtain on $\H_{T,0}^s(\Sigma)$
\begin{equation*}
\wt{\text{T}_\Sigma} (\Adm_\Sigma^\pm)
= \frac{1}{2} \Id + \chi_0 R^\pm (\text{K}^\pm \n^\pm \times) P^\pm + D^\pm,
\end{equation*}
where $\chi_0 R^\pm (\text{K}^\pm \n^\pm \times) P^\pm$ and $D^\pm$
are pseudo-differential operators of order $-1$ and $-\infty$ respectively. This concludes the proof.
\end{proof}

In the real case, $\Gamma_D$ is smooth
and thus the boundaries $\Gamma_D^\pm \cup \Sigma$
are both lipschitzian but not of $\caC^1$~regularity.
Therefore the operator $\K^\pm$ is no longer
a compact operator in $\H_T^s \left( \Gamma_D^\pm \cup \Sigma \right) $.
To study this case,
a first possibility is to come back to
the case of the theorem by distorting the boundaries $\Gamma_D^\pm \cup \Sigma$
in new boundaries $\wt{\Gamma_D^\pm} \cup \wt{\Sigma} $
such that these are $\caC^\infty$ (see \textsc{Fig.} \ref{picture4}, on the right),
and by introducing the cut-off function~$\chi_0$.
The theoretical analysis of these two approximations
(the change of boundaries and the multiplication by the smooth function $\chi_0$)
is not straightforward.
A more direct approach would be to extend
the theory of integral equations on surfaces with singularities.
Such a theory was developed for bi-dimensional Helmholtz problems in \cite{MolkoThesis},
but its extension to three-dimensional Maxwell problems remains to be done.

\section{Numerical results}
\label{Results}

In this section, we first describe the numerical discretization chosen for
the admittance operators $\Adm_\Sigma^+$ and $\Adm_\Sigma^-$.
We then explain the discretization of the preconditioning by the operator $\T_\Sigma$
of the equation \iref{DDM}, which couples the subdomains in our domain decomposition method. \\

We want to solve a numerical discretization of 
\begin{equation*}
(\Adm_\Sigma^+ + \Adm_\Sigma^-) \u = \u_0 \quad \text{ on } \Sigma.
\end{equation*}
We recall that
$\Adm_\Sigma^\pm \v_0 = R^\pm \v^\pm$ where $\v^\pm$ is solution of the EFIE:
$\T^\pm (\v^\pm) = \left( \frac{1}{2} \Id + \K^\pm \n^\pm \times \right) ( P^\pm \v_0)$.
To describe the action of the operator $\Adm_\Sigma^\pm$, one needs to discretize the EFIE. \\

We denote by $\Gamma_h^\pm$ a family of triangulations of $\partial \Omega^\pm = \Gamma_D^\pm \cup \Sigma$
such that $\Sigma_h := \Gamma_h^+ \cap \Gamma_h^-$ is a family of triangulations of $\Sigma$.
The space of $H_\text{div}$-conforming Rao-Wilton-Glisson finite elements
on $\Sigma_h$ is denoted by $X_h$, and we denote by $(\varphi_i)_{1 \leq i \leq N}$ its basis functions.
Similarly, the notation $X_h^\pm$ stand for the spaces of RWG finite elements on $\Gamma_h^\pm$,
associated with basis functions $(\psi_i^\pm)_{1 \leq i \leq N^\pm}$,
where $N^\pm > N$ and where we assume that $\forall i \leq N$, $\psi_i^\pm = \varphi_i$.
These assumptions will allow us below to apply the preconditioner to
a vector in the space $X_h$. \\

Let us describe the numerical computation of operator $\Adm_\Sigma^+$
applied to a vector $\v_0 = \sum_{i=1}^N v_{0,i} \varphi_i$ of $X_h$.
We first extend $\v_0$ to a vector $\v_0^+$ of $X_h^+$ defined by
\begin{equation*}
\v_0^+ = P^+ \v_0 =  \left( \sum_{i=1}^{N} v_{0,i} \varphi_i + \sum_{i=N+1}^{N^+} 0 \times \psi_i^+ \right) \in X_h^+,
\end{equation*}
and we denote by
\begin{equation*}
V_0^+ = (v_{0,1}, \ldots, v_{0,N}, 0, \ldots, 0 )^T
\end{equation*}
the vector of its components in the basis of $X_h^+$. \\

Given an operator $A$ and a space $X$ of RWG functions associated with a triangulation $T_0$,
we denote by $[A]_X$ its Galerkin matrix for the $L^2$-product using the basis functions $(\theta_i)_{i}$ of $X$,
namely $([A]_X)_{ij} = \int_{T_0} A \theta_i \cdot \theta_j$.
Consequently, the notation $[\T^+]_{X_h^+}$ stands for the Galerkin matrix of the single layer operator
defined on the triangulation $X_h^+$. \\

The EFIE can be discretized as follows :
\begin{equation*}
\text{Find } V^+ = (v_{1}, \ldots, v_{N^+})^T
\text{ such that }
[\T^+]_{X_h^+} V^+ = \left[ \frac{1}{2} \Id + \K^\pm \n^\pm \times \right]_{X_h^+} V_0^+.
\end{equation*}
and we have $\v^+ = \sum_{i=1}^{N^+} v_i \psi_i^+$.
The vector $R^+ \v^+$ is finally given by
$R^+ \v^+ = \sum_{i=1}^N v_i \varphi_i$. \\

From now on, we denote by $\left< \Adm^\pm_{X_h^\pm} \right>$ the numerical computation
of $\Adm_\Sigma^\pm$ described above.
Using the former finite elements, the equation \iref{DDM} takes the form of the linear system
\begin{equation*}
\left( \left< \Adm^+_{X_h^+} \right> + \left< \Adm^-_{X_h^-} \right> \right) U = U_0,
\end{equation*}
where $U = (u_1, \ldots, u_n)^T$ and $U_0 = (u_{0,1}, \ldots, u_{0,n})^T$,
with $\u = \sum_{i=1}^N u_i \varphi_i$ and $\u_0 = \sum_{i=1}^N u_{0,i} \varphi_i$. \\

To precondition the DDM, we have mathematically proposed a
multiplication by the operator $\T_\Sigma$.
Numerically speaking, one wants to obtain a linear system close to the identity matrix.
If we only multiplied the numerical vector by the matrix $[\T_\Sigma]_{X_h}$,
we would obtain a linear system close to the mass matrix $[\Id]_{X_h}$.
Therefore, we have to do a numerical multiplication by the matrix $[\Id]_{X_h}^{-1} [\T_\Sigma]_{X_h}$,
in order to solve a system close to the identity matrix, and then better conditioned.
As illustrated below,
preconditioning the method with the Galerkin matrix $[\T_\Sigma]_{X_h}$
is not enough to ensure an optimized convergence.
One also needs to inverse the system by the mass matrix on the interface,
which is realized through an iterative solution, of small numerical cost
thanks to the sparsity of the matrix $[\Id]_{X_h}$.
This operation converts a vector
whose coefficients are the $L^2$-scalar products with the basis functions,
into an amplitude vector
(a vector whose coefficients are the coordinates in the basis functions). \\

We denote by DDM~Y0 the original unpreconditioned equation
related to the linear system $(\Adm_\Sigma^+ + \Adm_\Sigma^-)$,
\begin{equation*}
\left( \left< \Adm^+_{X_h^+} \right> + \left< \Adm^-_{X_h^-} \right> \right) U = U_0.
\end{equation*}
DDM~Y1 is the equation with the left preconditioner
being the Galerkin matrix of the single layer operator,
\begin{equation*}
[\T_\Sigma]_{X_h} \left( \left< \Adm^+_{X_h^+} \right> + \left< \Adm^-_{X_h^-} \right> \right) U = [\T_\Sigma]_{X_h} U_0.
\end{equation*}
DDM~Y2 is the equation with the left preconditioner
being the Galerkin matrix of the single layer operator,
with an additional inversion by the mass matrix $[\Id]_{X_h}$,
\begin{equation*}
[\Id]_{X_h}^{-1} [\T_\Sigma]_{X_h} \left( \left< \Adm^+_{X_h^+} \right> + \left< \Adm^-_{X_h^-} \right> \right) U
= [\Id]_{X_h}^{-1} [\T_\Sigma]_{X_h} U_0.
\end{equation*}
DDM~Y3 is the equation with the right preconditioner
being the Galerkin matrix of the single layer operator,
and an inversion by the mass matrix $[\Id]_{X_h}$
\begin{equation*}
\left( \left< \Adm^+_{X_h^+} \right> + \left< \Adm^-_{X_h^-} \right> \right) [\Id]_{X_h}^{-1} [\T_\Sigma]_{X_h} U
= U_0.
\end{equation*}

Let us remark that we have to solve two kinds of linear systems.
The first one is the linear system arising from the DDM itself.
The second one is made of the systems which come from the discretization of the integral equations inside each subdomain.
In order to solve both of them, we use the GMRES algorithm.
Notice that the numerical scheme is a doubly nested iterative method.

\newpage
\subsection{Validation of the method}

First of all, we consider the degenerate case where
there is \emph{no scattering object}.
We denote by $\Sigma$ the sphere centered at the origin
and of diameter 1m, and we decompose $\bbR^3$ into two subdomains:
the interior and the exterior of $\Sigma$.
Notice that Lemma \ref{LemExact} applies to this situation.
The sphere meshed with 168 DoF is shown on \textsc{Fig.}~\ref{picture5} (left),
while the convergence curves are presented on
\textsc{Fig.}~\ref{picture5} (right)
at the frequency 68 MHz,
and for the four equations above.
Our first observation is that
the unpreconditioned DDM~Y0 converges really slowly
in comparison with the three preconditioned DDM~Y1, DDM~Y2 and DDM~Y3.
DDM~Y2 converges faster than DDM~Y1,
which lacks the inversion by the mass matrix.
The results obtained with the right preconditioner (DDM~Y3)
are comparable with those of the left preconditioner (DDM~Y2):
both converge in as few as 4 iterations.

\begin{figure}[h!]
\begin{minipage}{0.26\linewidth}
\includegraphics[scale=0.20]{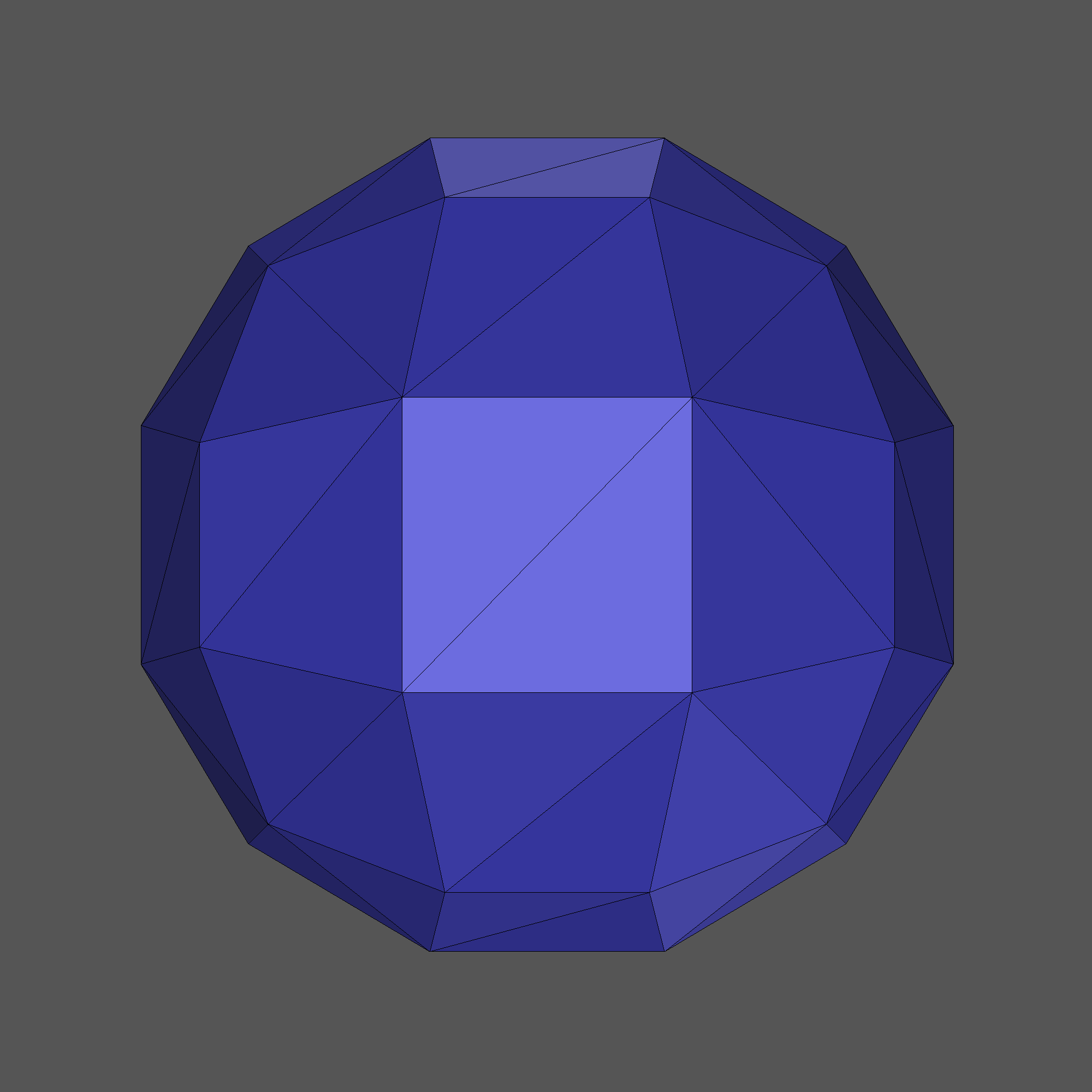}
\end{minipage}
\hfill
\begin{minipage}{0.70\linewidth}
\includegraphics[scale=0.40]{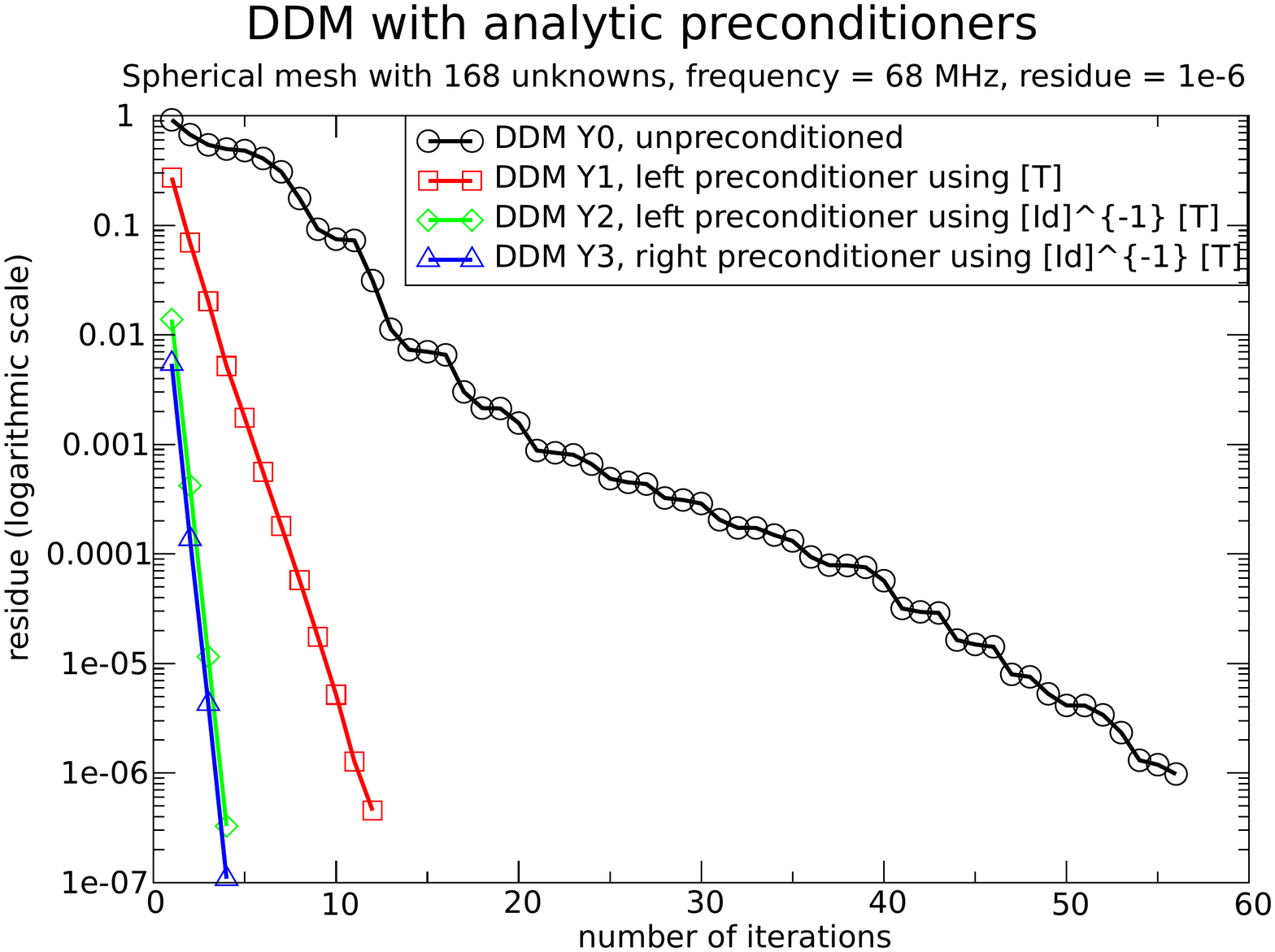}
\end{minipage}
\caption{
Mesh (left) and convergence curves (right) to reach a residue of order $10^{-6}$,
for the artificial sphere at 68 MHz meshed with 168 DoF. }
\label{picture5}
\end{figure}

In the next experiment, there is a scattering object which
contains a cavity. This is the original setting intended for our study.
First, we present the case of an object whose shape is close
to a parallelepipedic box which is open at one of its extremities (\textsc{Fig.} \ref{picture18}),
and therefore exhibits a cavity.
The interface $\Sigma$ of this parallelepipedic box is a flat rectangle and is meshed with 102 DoF.
For a residue of order $10^{-6}$,
DDM~Y1, DDM~Y2 and DDM~Y3 converge respectively 
in 19, 13 and as few as 11 iterations,
whereas the unpreconditioned method has not reached convergence after 1000 iterations.

\newpage
\begin{figure}[h!]
\begin{minipage}{0.26\linewidth}
\includegraphics[scale=0.20]{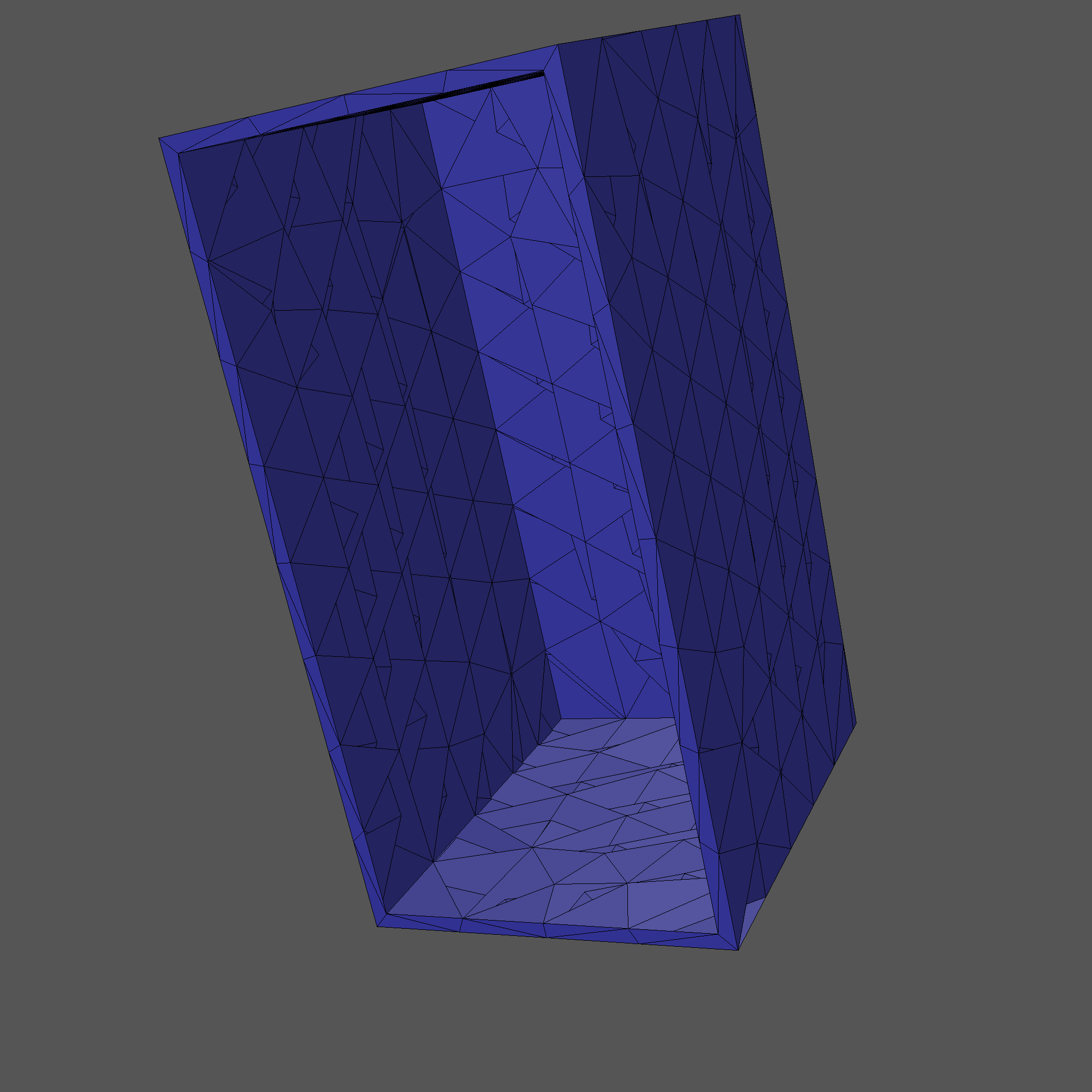}
\end{minipage}
\hfill
\begin{minipage}{0.70\linewidth}
\includegraphics[scale=0.40]{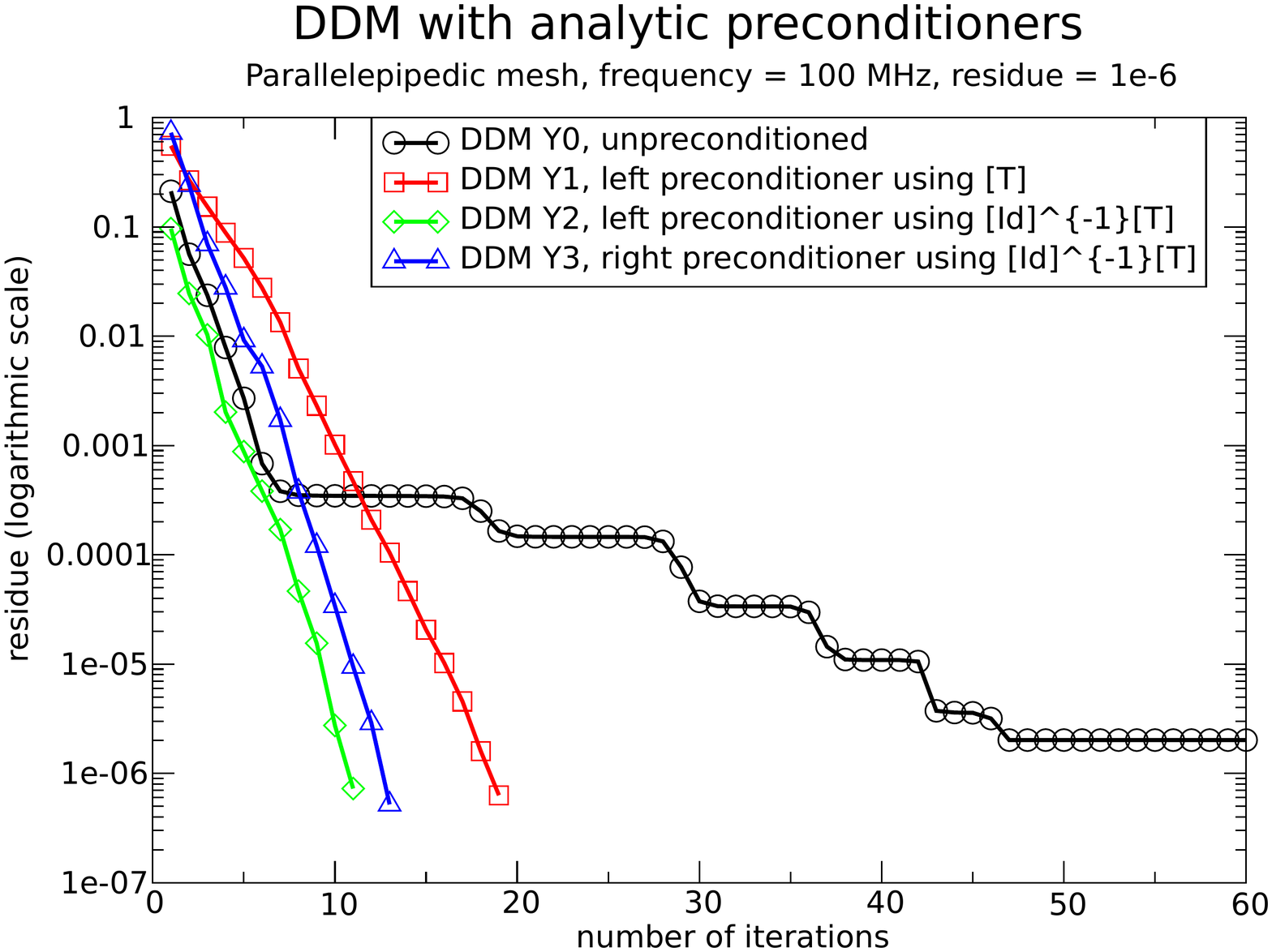}
\end{minipage}
\caption{Mesh (left) and convergence curves (right) to reach a residue of order $10^{-6}$,
for the parallelepipedic box at 100 MHz meshed with 102 DoF on the interface $\Sigma$. }
\label{picture18}
\end{figure}

On \textsc{Fig.} \ref{pictureSER},
we compare the radar cross section (RCS)
obtained by the four methods to the one obtained with
the integral equation EFIE on the global mesh,
without any artificial interface.
Although this case is quite simple, there is no artefact due to the method.

\begin{figure}[h!] \centering
\begin{minipage}[t]{0.5\linewidth}
\includegraphics[scale=0.32]{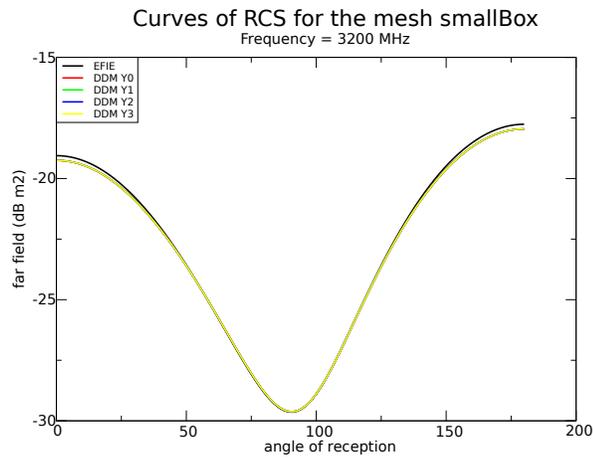}
\end{minipage}
\caption{Curves of RCS for the mesh smallBox at a frequency of 3200 MHz. }
\label{pictureSER}
\end{figure}

\newpage
\subsection{Reliability of the method with respect to the frequency}

Our third experiment illustrates the influence of the frequency increase on the convergence rate,
for the sphere again, but with a finer mesh
of the spherical interface $\Sigma$, of 3072 DoF.
We choose to compare only DDM~Y0 (without preconditioner)
with DDM~Y2 (with the left preconditioner).
The number of iterations to reach a residue of order $10^{-5}$
increases with the frequency for DDM~Y0,
whereas it remains stable (always 4 iterations) for DDM~Y2 (\textsc{Fig.}~\ref{picture6} and \textsc{Tab.}~\ref{picture7}).
Consequently, the convergence rate is not altered by the increase of the frequency,
as illustrated on \textsc{Fig.}~\ref{picture5}, on the right, and on \textsc{Fig.}~\ref{picture6}.

\begin{figure}[h!]
\begin{minipage}{0.26\linewidth}
\includegraphics[scale=0.20]{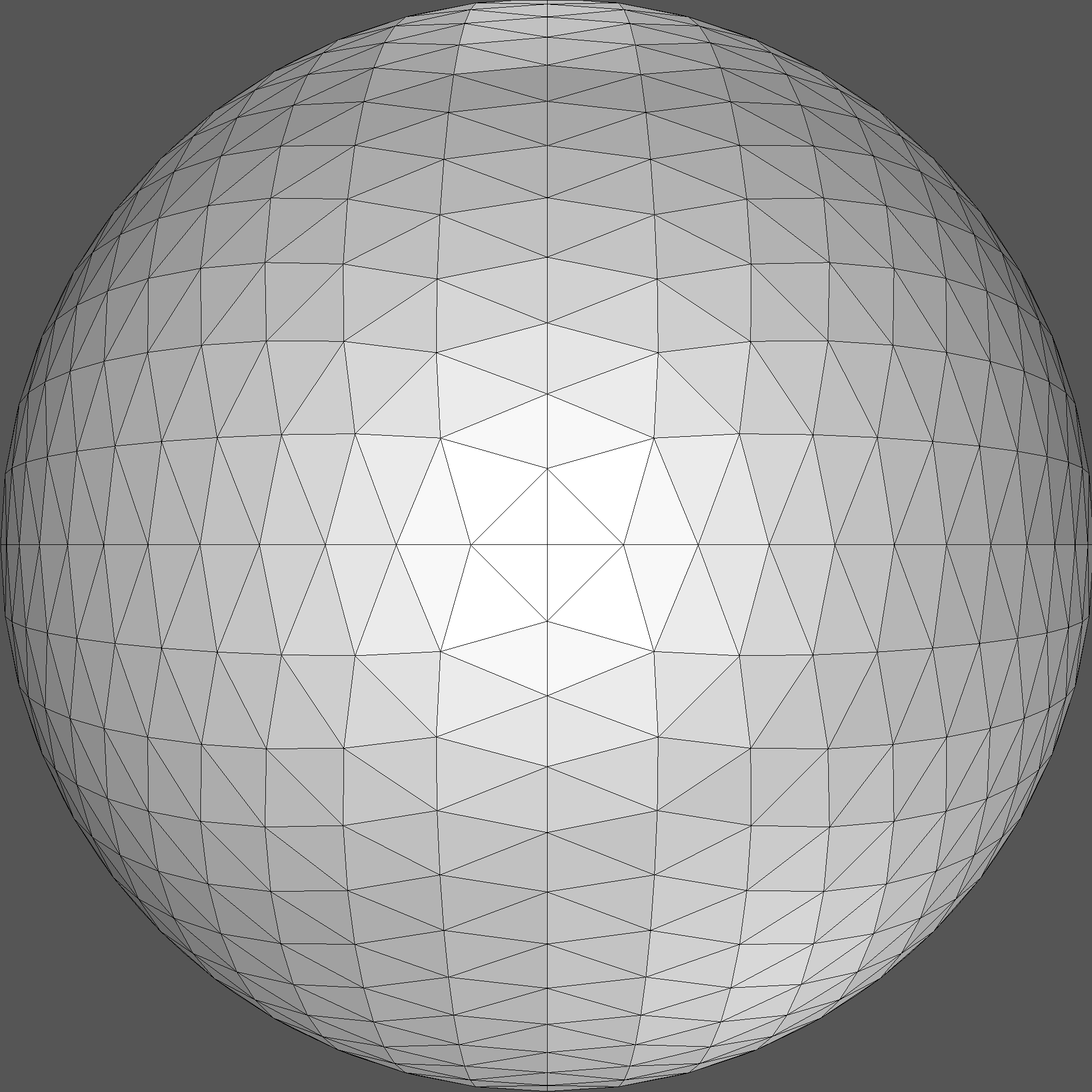} 
\end{minipage}
\hfill
\begin{minipage}{0.70\linewidth}
\includegraphics[scale=0.40]{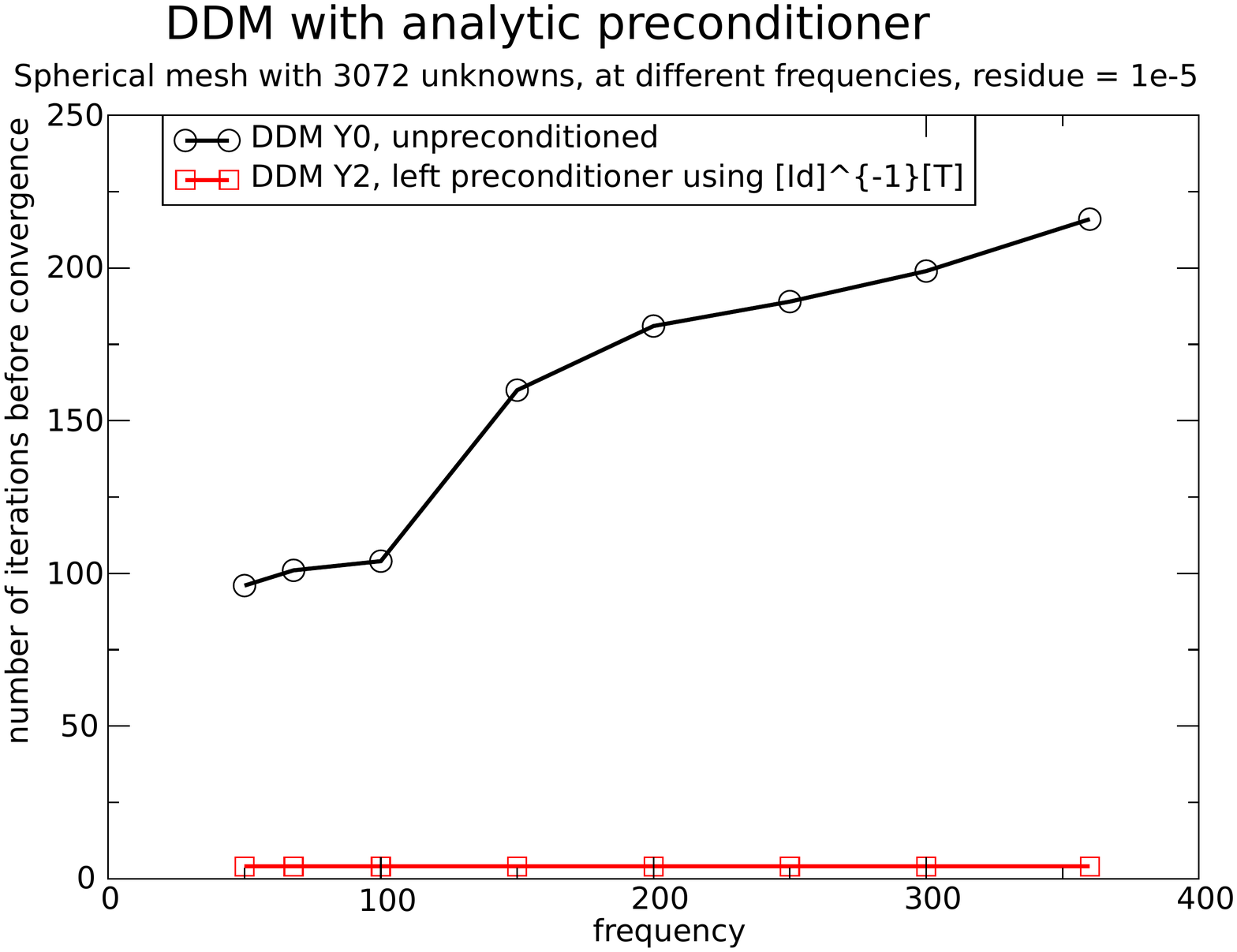}
\end{minipage}
\caption{
Mesh (left) and influence of the frequency increase (right)
on the number of iterations to reach a residue of order $10^{-5}$,
for the artificial sphere meshed with 3072 DoF. }
\label{picture6}
\end{figure}

\begin{table}[h!] \centering
\begin{tabular}{|c|c|c|c|c|c|c|c|c|}
\hline
Frequency (MHz) & 50 & 68 & 100 & 150 & 200 & 250 & 300 & 360 \\
\hline
DDM~Y0 & 96 & 101 & 104 & 160 & 181 & 189 & 199 & 216 \\
\hline
DDM~Y2 & 4 & 4 & 4 & 4 & 4 & 4 & 4 & 4  \\
\hline
\end{tabular} \\
\caption{Iterations count to reach a residue of order $10^{-5}$
depending on the frequency, for the sphere with 3072 DoF.}
\label{picture7}
\end{table}

\newpage
\subsection{Reliability of the method with respect to the number of unknowns}

\subsubsection{Artificial spheres (no real object)}
\;\ \\

We now refine the mesh of the sphere, passing from 3072 DoF to 5292 DoF.
On this spherical mesh, at a frequency of 400 MHz, and to reach a residue of $10^{-4}$,
DDM~Y2 converges in 28 iterations.
The condition number of the linear system has obviously increased
in comparison with the one of the former mesh (with 3072 DoF)
(see \textsc{Tab.}~\ref{pictureTableSp3528}), leading to a smaller convergence rate.

\begin{figure}[h!] \centering
\includegraphics[scale=0.32]{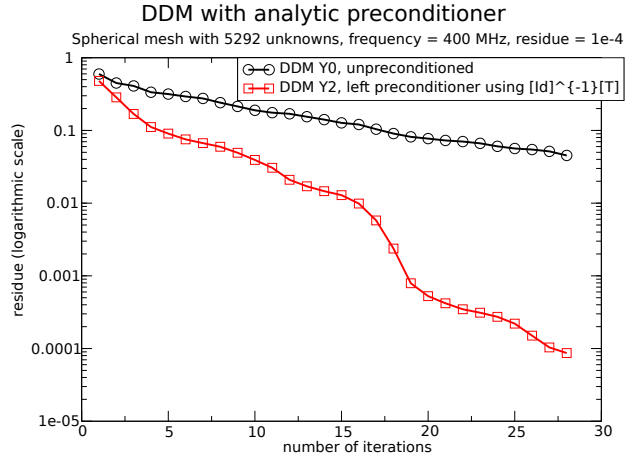}
\caption{Convergence curves to reach a residue of order $10^{-4}$,
for the spherical mesh with 5292 DoF, at a frequency of 400~MHz,
respectively for DDM~Y0 (unpreconditioned)
and for DDM~Y2 (with analytic preconditioner). }
\label{pictureSp3528}
\end{figure}

\begin{table}[h!] \centering
\begin{tabular}{|c|c|c|c|c|c|c|c|}
\hline
Equation & Frequency & Residue & Number of unknowns & Number of iterations  \\
\hline DDM~Y2 & 360 MHz & $10^{-5}$ & 3072 & 4 \\
\hline DDM~Y2 & 400 MHz & $10^{-4}$ & 5292 &  28 \\
\hline
\end{tabular}
\caption{For the algorithm DDM~Y2, comparison between the number of iterations
needed to reach a given residue at a given frequency,
respectively for the spherical mesh with 3072~DoF
and for the spherical mesh with 5292~DoF. }
\label{pictureTableSp3528}
\end{table}

Nevertheless, looking at the convergence curves of the residues 
with respect to the iterations (\textsc{Fig.}~\ref{pictureSp3528}),
we observe that the unpreconditioned DDM~Y0 converges much slower than DDM~Y2.
In particular, after 28 iterations, DDM~Y0 has not reached a residue of $5. 10^{-2}$,
whereas DDM~Y2 has reached a residue of $10^{-4}$.
As a conclusion, this preconditioner remains very efficient for finer geometries.

\newpage
\subsubsection{Hollow spheres (real objects)}
\;\ \\

In this section, we illustrate the behavior of the algorithms (DDM~Y0 to DDM~Y3) when we refine 
the mesh of the scattering object.
In that purpose, we consider a hollow sphere which constitutes the real scattering obstacle.
The sphere is of radius 1 meter and is open for latitudes higher than 45 degrees,
and is discretized with six different meshes of increasing precision
(see \textsc{Fig.}~\ref{pictureHollow35} for the most refined mesh).

The artificial interface needed for the DDM algorithm
is chosen to be the missing cap of the sphere.
Therefore, the interior and exterior problems consist in solving Maxwell equations inside and outside the sphere, respectively.
They exchange data on the cap while, on the rest of the sphere, we have a Dirichlet type boundary condition.
We give in \textsc{Tab.}~\ref{pictureTableMeshHollow}
the number of unknowns respectively on the interface
and on the whole sphere, for each considered mesh.
Due to the size of the meshes, and contrarily to what has been done so far,
we use a fast multipole method (FMM) to compress all involved linear systems.

\begin{table}[h!] \centering
\begin{tabular}{|c|c|c|c|c|c|c|c|}
\hline
Name of the mesh & Number of DoF & Number of DoF on the spherical mesh \\
& on the interface (cap) & of both subdomains \\
\hline hollow12 & 888 & 5184 \\
\hline hollow15 & 1380 & 8100 \\
\hline hollow20 & 2440 & 14400 \\
\hline hollow25 & 3800 & 22500 \\
\hline hollow30 & 5460 & 32400 \\
\hline hollow35 & 7420 & 44100 \\
\hline
\end{tabular}
\caption{Meshes of the considered hollow spheres. }
\label{pictureTableMeshHollow}
\end{table}

\newpage

\begin{figure}[h!] \centering
\includegraphics[scale=0.34]{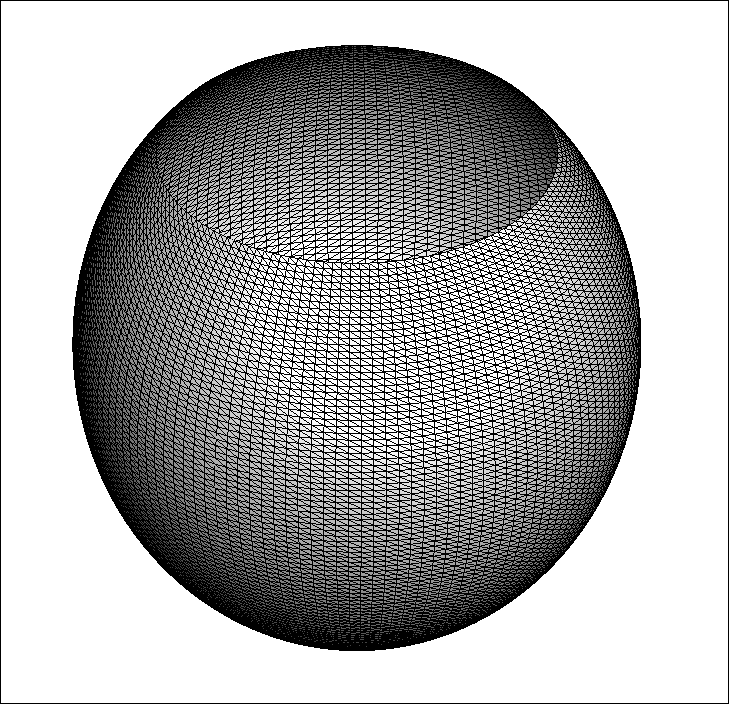}
\caption{Mesh (hollow35) of the hollow sphere of radius 1m.
The interface (not represented) possesses 7420 DoF while the whole sphere has 44100 DoF. }
\label{pictureHollow35}
\end{figure}

\begin{figure}[h!] \centering
\includegraphics[scale=0.42]{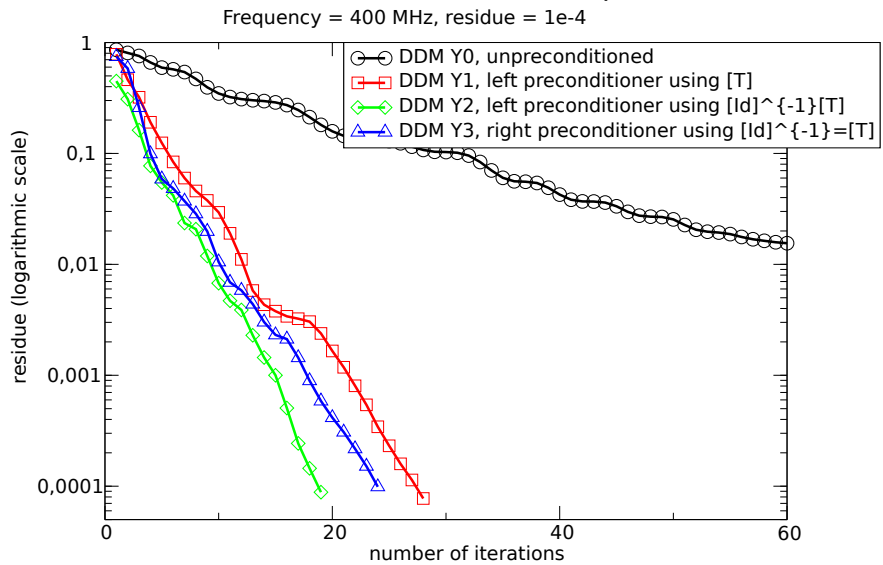}
\caption{Convergence curves to reach a residue of order $10^{-4}$,
for the mesh hollow35 of the hollow sphere, at a frequency of 400~MHz,
respectively for DDM~Y0 (unpreconditioned) and DDM~Y1, DDM~Y2, DDM~Y3 (with analytic preconditioners). }
\label{pictureCurveHollow35}
\end{figure}

\newpage

\begin{figure}[h!] \centering
\includegraphics[scale=0.45]{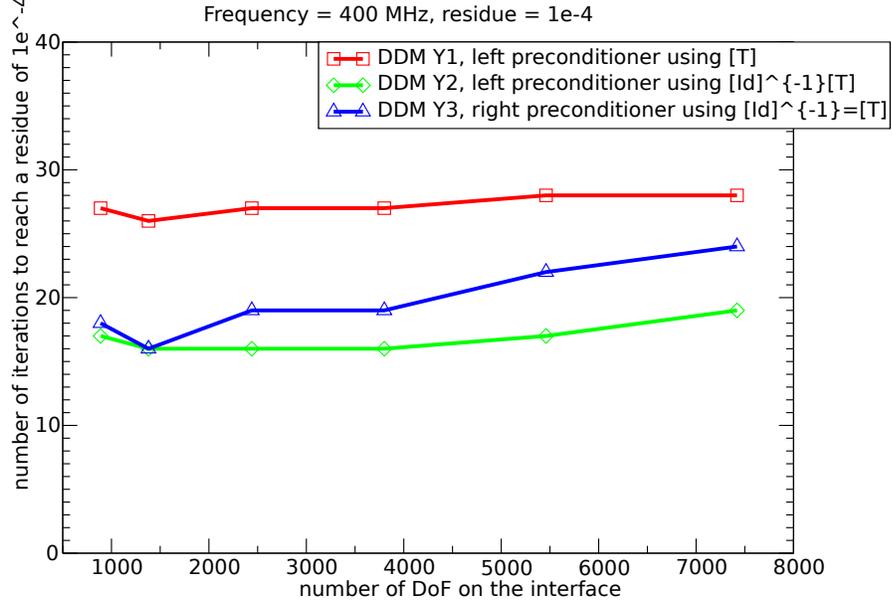}
\caption{
Comparison of the number of iterations for the preconditioned equations DDM~Y1, DDM~Y2, DDM~Y3,
to reach a residue of $10^{-4}$, at a constant frequency of 400~MHz.
For each curve, each point corresponds to one of the different meshes
that we have considered, namely with 888, 1380, 2440, 3800, 5460 and 7420 DoF
on the interface. }
\label{pictureCompIter}
\end{figure}

We show in \textsc{Fig.}~\ref{pictureCurveHollow35} the convergence rates
for the four methods DDM~Y0 to DDM~Y3 in the finest case (hollow35),
for a frequency of 400~MHz.
Once again, the unpreconditioned DDM~Y0 converges much slower than the three other
preconditioned equations (DDM~Y1, DDM~Y2, DDM~Y3).
For instance, DDM~Y1 (resp. DDM~Y2, DDM~Y3) reaches a residue of $10^{-4}$ in 28 iterations (resp. 19, 24 iterations),
whereas DDM~Y0 has not yet reached a residue of $10^{-2}$ in 60 iterations.
The explicit numbers of iterations for all meshes
are provided in \textsc{Tab.}~\ref{pictureTableConvergenceHollow}.

\begin{table}[h!] \centering
\begin{tabular}{|l|c|c|c|c|c|c|c|}
\hline
Equation & DDM~Y0 & DDM~Y1 & DDM~Y2 & DDM~Y3 \\
Mesh & & & & \\
\hline hollow12 & 152 & 27 & 17 & 18 \\
\hline hollow15 & 186 & 26 & 16 & 16 \\
\hline hollow20 & $> 60$ & 27 & 16 & 19 \\
\hline hollow25 & $> 60$ & 27 & 16 & 19 \\
\hline hollow30 & $> 60$ & 28 & 17 & 22 \\
\hline hollow35 & $> 60$ & 28  & 19  & 24 \\
\hline
\end{tabular}
\caption{
Number of iterations to reach a residue of order $10^{-4}$, at a constant frequency of 400~MHz,
for the six meshes of the hollow sphere,
for DDM~Y0 (unpreconditioned) and DDM~Y1, DDM~Y2, DDM~Y3 (with analytic preconditioners). }
\label{pictureTableConvergenceHollow}
\end{table}

Increasing the frequency of the problem to 1~GHz,
only for the finest mesh (hollow35),
does not deteriorate the method that much.
Indeed, we show in \textsc{Fig.}~\ref{pictureCurveHollow35_f1G}
and \textsc{Tab.}~\ref{pictureTableConvergenceHollow35_f1G}
that the three preconditioned methods DDM~Y1, DDM~Y2, DDM~Y3
remain very competitive in comparison with DDM~Y0.

\begin{figure}[h!] \centering
\includegraphics[scale=0.42]{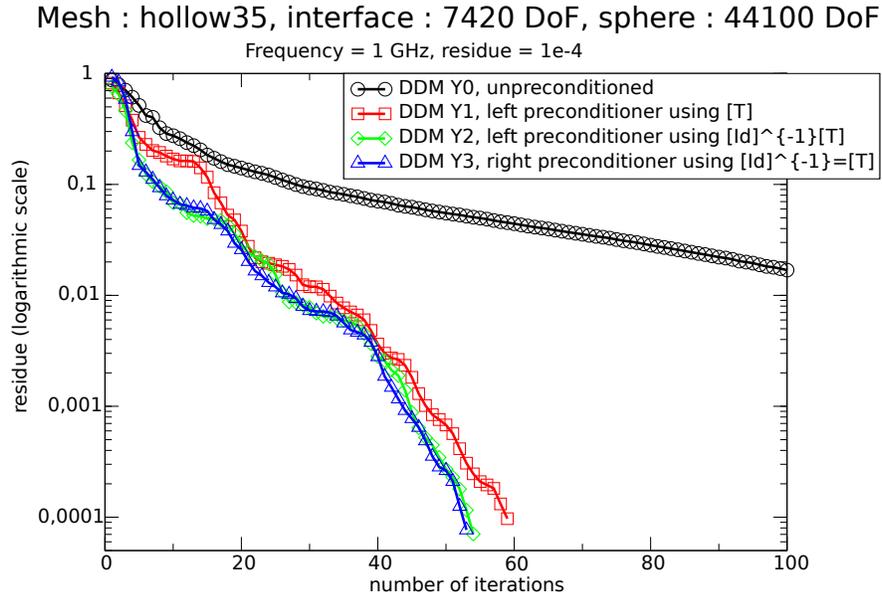}
\caption{Convergence curves to reach a residue of order $10^{-4}$,
for the mesh hollow35 of the hollow sphere, whose interface is meshed with 7420 DoF, at a frequency of 1~GHz,
respectively for DDM~Y0 (unpreconditioned) and DDM~Y1, DDM~Y2, DDM~Y3 (with analytic preconditioners). }
\label{pictureCurveHollow35_f1G}
\end{figure}

\begin{table}[h!] \centering
\begin{tabular}{|l|c|c|c|c|c|c|c|}
\hline
Equation & DDM~Y0 & DDM~Y1 & DDM~Y2 & DDM~Y3 \\
Mesh & & & & \\
\hline hollow35 & $> 100$ & 59  & 54  & 53 \\
\hline
\end{tabular}
\caption{
Number of iterations to reach a residue of order $10^{-4}$, at a frequency of 1~GHz,
for the mesh hollow35, whose interface is meshed with 7420 DoF,
for DDM~Y0 (unpreconditioned) and DDM~Y1, DDM~Y2, DDM~Y3 (with analytic preconditioners). }
\label{pictureTableConvergenceHollow35_f1G}
\end{table}

\newpage
\section{Conclusion}

We have proposed a domain decomposition method associated with an efficient preconditioner,
based on the restriction of the single layer operator
on the interface between the subdomains.
In each subdomain, the EFIE is solved at each iteration.
The numerical results illustrate the very good behavior of the resulting preconditioned algorithm,
which converges much faster than without preconditioning. \\

Nevertheless, although the proposed method seems very encouraging, several difficulties need
still to be overcome in order to make the method usable in real applications.
First, the present formulation is restricted to the case where the EFIE is solved in each subdomain.
Clearly, there is an obvious obstruction for resonant frequencies. In order to circumvent this issue,
we have to generalize the approach for other formulations (e.g. CFIE, or the recent very efficient
GCSIE methods \cite{Alouges07}, \cite{AlougesWaves07}). \\

Another improvement direction consists in changing the coupling condition on the surface $\Sigma$ 
between the subdomains. For instance, when one takes impedant coupling boundary conditions, it 
is well known that the underlying problems are well-posed for any frequency \cite{CakoniColtonMonk04}.
Again, the GCSIE formalism, originally developed for metallic problems,
has recently been extended to impedant ones in \cite{LevadouxMillotPernet} and \cite{Pernet},
and could prove to be very efficient. \\

We plan to investigate those issues and even combinations of them in the foreseeing future.

\subsection*{Acknowledgements}
We would like to address special thanks to Jean-Marie Mirebeau for helpful suggestions.


\end{document}